\documentclass[11pt]{article}
\title{\vskip-2.0em Hochschild homology and cohomology of $\lp{1}(\Z_+^k)$}
\author{Yemon Choi}
\date{2nd September 2008}

\usepackage{amsmath,amsfonts}
\usepackage{hyperref}

\newcommand{\ackndiag}{Uses Paul Taylor's {\tt diagrams.sty} macros.}
\RequirePackage[PS,nohug]{diagrams}

\newcommand{\LR}[2]{\pile{\lTo^{#1} \\ \rTo_{#2} }} 
\newcommand{\RL}[2]{\pile{\rTo^{#1} \\ \lTo_{#2} }} 

\RequirePackage{amsthm}   

\newcounter{pulse}[section]
\numberwithin{pulse}{section}  

\theoremstyle{plain}
\newtheorem{thm}[pulse]{\sc Theorem}
\newtheorem{propn}[pulse]{\sc Proposition}
\newtheorem{lemma}[pulse]{\sc Lemma}
\newtheorem{coroll}[pulse]{\sc Corollary}
\theoremstyle{definition}
\newtheorem{defn}[pulse]{\sc Definition}
\newtheorem*{notn}{{\sc Notation}}

\theoremstyle{remark}
\newtheorem{rem}[pulse]{\sc Remark}
\newtheorem*{remstar}{\sc Remark}

\newcommand{\veps}{\varepsilon}
\newcommand{\wtild}{\widetilde}
\newcommand{\blob}{\bullet}
\newcommand{\blank}{\underline{\quad}}
\newcommand{\iso}{\cong}
\newcommand{\miso}{\underset{1}{\iso}} 
\newcommand{\sid}{{\mathop{\sf id}}}

\DeclareMathOperator{\Ker}{ker}
\DeclareMathOperator{\Coker}{coker}
\DeclareMathOperator{\Image}{im}

\newcommand{\alg}{{\sf alg}}  

\newcommand{\pcat}[1]{{\sf {#1}}} 

\newcommand{\st}{\;:\;}
\newcommand{\defeq}{:=}
\newcommand{\dt}[1]{{\it #1}\/}  
\renewcommand{\emph}[1]{{\sl #1\/}}  
\renewcommand{\em}{\sl}   

\newcommand{\Ban}{\pcat{Ban}}

\newcommand{\Mod}[2]{{}_{#1}{\sf mod}_{#2}} 
\newcommand{\unMod}[2]{{}_{#1}{\sf unmod}_{#2}} 

\newcommand{\LMod}[1]{\Mod{#1}{}}  
\newcommand{\RMod}[1]{\Mod{}{#1}} 
\newcommand{\LRMod}[2]{\Mod{#1}{#2}} 

\newcommand{\LunMod}[1]{\unMod{#1}{}} 
\newcommand{\RunMod}[1]{\unMod{}{#1}} 
\newcommand{\LRunMod}[2]{\unMod{#1}{#2}} 

\newcommand{\lHom}[1]{{}_{#1}{\rm Hom}} 

\newcommand{\abs}[1]{\vert{#1}\vert}

\newcommand{\norm}[1]{\Vert{#1}\Vert}

\newcommand{\Lin}[2]{{\mathcal L}({#1},{#2})} 

 \newcommand{\cvct}[2]{\left(\begin{matrix}{#1}\\{#2}\end{matrix}\right)} 

\newcommand{\pair}[2]{\langle{#1},\,{#2}\rangle}

\newcommand{\clos}[1]{\overline{#1}}
\newcommand{\lin}{\mathop{\rm lin}\nolimits} 

\newcommand{\tp}{{\scriptstyle\otimes}}
\newcommand{\ptp}{{\scriptstyle\widehat{\otimes}}}
\newcommand{\bigptp}{\widehat{\bigotimes}}

\newcommand{\ptpR}[1]{{\underset{#1}{\scriptstyle\widehat{\otimes}}}}

\newcommand{\Cplx}{\mathbb C}
\newcommand{\Nat}{\mathbb N}
\newcommand{\Rat}{\mathbb Q}

\newcommand{\Z}{{\mathbb Z}}
\newcommand{\Sym}[1]{S_{#1}} 

\newcommand{\lp}[2][]{\ell^{#2}_{#1}}
\newcommand{\Lp}[1]{L_{#1}}

\newcommand{\id}[1][]{{\sf 1}_{#1}} 
\newcommand{\fu}[1]{{#1}^{\sharp}}    
\newcommand{\cu}[1]{{#1}_{\sf un}}  

\newcommand{\qKahl}[1]{\widetilde{\Omega}_{#1}}  

\newcommand{\BGS}[2][]{e_{#1}(#2)}
\newcommand{\coBGS}[2][]{e_{#1}(#2)^*}

\newcommand{\fA}{{\mathfrak A}}
\newcommand{\sA}{{\sf A}}
\newcommand{\sB}{{\sf B}}
\newcommand{\sR}{{\sf R}}

\DeclareMathOperator{\preAss}{\widetilde{\sf Ass}}
\DeclareMathOperator{\preEx}{\widetilde{\sf Ex}}
\DeclareMathOperator{\Ass}{{\sf Ass}}
\DeclareMathOperator{\Ex}{{\sf Ex}}

\newcommand{\algExt}{{\sf Ext}}
\newcommand{\Tor}{{\rm Tor}}
\newcommand{\Ext}{{\rm Ext}}
\newcommand{\Ho}[3][]{{\mathcal {#2}_{#3}^{#1}}}  
\newcommand{\Co}[3][]{{\mathcal {#2}^{#3}_{#1}}} 

\newcommand{\HarH}[3][]{{\mathcal {H}ar}{\mathcal {#2}}_{#3}^{#1}}
\newcommand{\HarC}[3][]{{\mathcal {H}ar}{\mathcal {#2}}^{#3}_{#1}}

\newcommand{\bdy}{{\sf d}}
\newcommand{\dif}{\delta} 
\newcommand{\face}[2][]{\partial^{#1}_{#2}} 

\newcommand{\Kahler}{K\"{a}hler}
\newcommand{\Kunneth}{K\"{u}nneth}

\newcommand{\al}{\alpha}

\newcommand{\gm}{\gamma}

\newcommand{\sig}{\sigma}
\newcommand{\tht}{\theta}

\numberwithin{equation}{section}

\begin{document}
\maketitle
\begin{abstract}
Building on the recent determination of the simplicial cohomology groups of the convolution algebra ${\ell}^1({\mathbb Z}_+^k)$ [Gourdeau, Lykova, White, 2005] we investigate what can be said for cohomology of this algebra with more general symmetric coefficients. Our approach leads us to a discussion of Harrison homology and cohomology in the context of Banach algebras, and a development of some of its basic features. As an application of our techniques we reprove some known results on second-degree cohomology.
\end{abstract}


\section{Introduction}
In the development of cohomology theories for $K$-algebras (where $K$ is a commutative ring), the polynomial rings $K[x_1,\ldots,x_k]$ have played an important role: not only as examples whose Hochschild homology and cohomology is completely understood, but as `free objects' which one can use to take resolutions of more complicated and interesting algebras. For example, the equivalence of Harrison and Andr\'e-Quillen cohomology over fields of characteristic zero relies crucially on knowing the structure of the cohomology of polynomial rings.

One would like to make similar computations and constructions in the Banach algebraic setting, for suitable completions of $\Cplx[x_1,\ldots,x_k]$. However, progress here has been much slower, and indeed the Banach setting produces new phenomena. For instance: $\Cplx[z]$ is known to have global dimension $1$\/, i.e.~it has vanishing cohomology in degrees $2$ and above for \emph{arbitrary coefficient modules}\/; yet it has long been
known that $\Co{H}{2}(\lp{1}(\Z_+),\lp{1}(\Z_+))$ is nonzero, and is in fact an infinite-dimensional Banach space, see~\cite{DaDunc}.
Thus even if we restrict to symmetric coefficients, complications may arise.

It was shown recently in \cite{GouJohWh} that the simplicial cohomology of the convolution algebra $\lp{1}(\Z_+)$ vanishes in degrees 2 and above. 
This tells us that our choice of coefficient module is important. The underlying aim of this paper is to see how much we can deduce, from knowledge of simplicial cohomology, about cohomology with general symmetric coefficients.

More precisely, we show that one can
\begin{itemize}
\item[$(i)$] deduce partial results on the cohomology of $\lp{1}(\Z_+^k)$ with symmetric coefficients from knowledge of cohomology of $\lp{1}(\Z_+)$ with symmetric coefficients; and
\item[$(ii)$] reduce the caculation of the cohomology groups $\Co{H}{n}(\lp{1}(\Z_+),M)$\/, where $M$ is a \emph{symmetric bimodule}, to knowledge of the properties of $M$ as a \emph{one-sided} module.
\end{itemize}
These results rely crucially on results from \cite{GouJohWh,GouLyWh}: our approach is to build on the results rather than try to generalise their proofs, by using machinery from homological algebra and ideas from the `Hodge decomposition' of Hochschild homology \cite{GS_Hodge}.

\subsection*{Overview of the paper}
The main results of this paper are Theorems \ref{t:only_Harr} and~\ref{t:Harr_of_Ak}, in the sense that the previous sections are directed towards their proof. We have nevertheless sought to work in slightly greater generality when setting up the preliminary results of Sections \ref{s:Harrison} and \ref{s:babyKunn}.

Although we are motivated by well-established results in commutative algebra, much of the machinery from that setting is simply not applicable in the Banach algebraic setting. We are therefore forced to develop some machinery from scratch, although in some cases we can adapt existing tools from commutative algebra with relative ease: this absence of precise analogues for algebraic tools is reflected in the length of the paper.
It is hoped that the partial results given here will encourage the refinement and extension of the crude tools of Sections \ref{s:Harrison} and \ref{s:babyKunn}\/.

\subsection*{Acknowledgments}
Apart from Section \ref{s:someH2}, the results here are taken (with a few modifications) from the author's PhD thesis \cite{YC_PhD}, which was supported by an EPSRC grant. The author acknowledges the support of the University of Newcastle-upon-Tyne while this work was done, and the support of the University of Manitoba while this article was written up.
He also thanks N.~Gr{\o}nb{\ae}k, Z.~A.~Lykova and M.~C.~White for helpful discussions.

\medskip\noindent{\bf Remark/correction added August 2008.}
It has been pointed out to the author that there is a slight problem with the hypotheses in Proposition~\ref{p:CBA_Hoch}, Coroll\-ary~\ref{c:BGS_arb-coeff} and Propo\-sition~\ref{p:Harr-as-Ext}. Namely, the proofs given for these results seem to require the symmetric module $M$ to be unit-linked as a $\cu{A}$-module. (Otherwise, it is not clear that tensoring or homming with $M$ behaves in the right way).

One can fix this gap either by working throughout with the forced unitization~$\fu{A}$ (Theorem~\ref{t:BGS_alg}
 is unaffected by such a change), or by requiring all modules to be $\cu{A}$-unit-linked. The former approach is much more natural, but has the slight disadvantage in the present context that we then need to have vanishing theorems for the homology groups $\Ho{H}{n}(\lp{1}(\Z_+),\fu{\lp{1}(\Z_+)})$, which can be deduced from the results of \cite{GouLyWh} but seem not to be stated explicitly there. The second approach sidesteps this issue, but limits the class of modules that we can consider.

Either fix is straightforward to implement: I have not done so here, mainly to preserve this document's status as a ``pre-publication'' article. The slip-up has been corrected for the published version, which has been accepted (August 2008) by the {\it Quarterly Journal of Mathematics}\/.

\begin{section}{Preliminaries and notation}

\begin{subsection}{General notation and terminology}\label{ss:genprelim}

Throughout this article we abuse notation and write $\sid$ for the \dt{identity map} on a set, vector space, module, and so on. It should always be clear from context what the domain of $\sid$~is.

\subsubsection*{Algebras with and without identity}
Although our eventual focus will be on the Banach algebras $\lp{1}(\Z_+^k)$, which are unital, some of the general machinery applies to algebras without an identity element. Some notation will be needed.

\begin{notn}
If $K$ is a commutative ring with identity, and $\sA$ is a $K$-algebra which may or may not possess an identity element, we can form the \dt{forced unitisation} of $\sA$, which will be denoted by $\fu{\sA}$\/. (In the case where $K=\Cplx$ and $\sA$ is a Banach algebra, $\fu{\sA}$ is also a Banach algebra.) 

We define the \dt{conditional unitisation of $A$}, denoted by $\cu{\sA}$, to be $\sA$ itself if $\sA$ has an identity element, and $\fu{\sA}$ otherwise.

If $\sB$ is a $K$-algebra with identity, then we shall usually denote its identity element by $\id$\/, or by $\id[\sB]$ if there is possible confusion over which algebra we are dealing with.
\end{notn}

\subsubsection*{Seminormed and Banach spaces}
The Hochschild homology and cohomology groups of a Banach algebra are in general seminormed, rather than normed, spaces. At several points in Section \ref{s:babyKunn} we want to assert that two given seminormed spaces are `isomorphic', and so we briefly make precise what `isomorphism' means in this context.

If $(V,\norm{\blank})$ is a seminormed vector space then we shall always equip it with the canonical topology that is induced by the pseudometric $(x,y)\mapsto \norm{x-y}$. Note that this topology need not be Hausdorff; indeed, it is Hausdorff if and only if $\{0\}$ is a closed subset of~$V$. Quotienting $V$ out by the subspace $\{x\in V\st \norm{x}=0\}$, we obtain a normed space which we refer to as the \dt{Hausdorff\-ifi\-cation of $V$}.

Just as for normed spaces, a bounded linear map between seminormed spaces is continuous. It follows that if $E$ and $F$ are seminormed spaces and there exist bounded linear, mutually inverse maps $S:E \to F$ and $T:F \to E$, then $E$ and $F$ are not just isomorphic as vector spaces but are homeomorphic as topological spaces.

\begin{defn}\label{dfn:Sns_iso}
Let $E$, $F$, $S$ and $T$ be as above. We say that $E$ and $F$ are \dt{isomorphic as seminormed spaces}, and that $S$ and $T$ are \dt{isomorphisms of seminormed spaces}\/.

In the case where $S$ and $T$ can be chosen to be \dt{isometries}, we shall (following \cite{DefFlor}) write $E \miso F$\/.
\end{defn}

\begin{remstar}
The point of labouring this definition is that \emph{a continuous linear bijection from one seminormed space onto another need not be a homeomorphism}, even when both spaces are complete. An easy -- albeit artificial -- example is provided by the identity map $\iota: (V,\norm{\blank}) \to (V, \norm{\blank}_0)$, where $\norm{\blank}$ is a not-identically-zero seminorm on $V$ and $\norm{\blank}_0$ denotes the zero seminorm; clearly $\iota$ is norm-decreasing and hence is continuous, but it cannot be a homeomorphism since the topology induced by $\norm{\blank}_0$ is the indis\-crete one.
\end{remstar}

\begin{notn}
If $E$ and $F$ are \emph{Banach spaces} then we shall denote the \dt{projective tensor product} of $E$ and $F$ by $E\ptp F$. For the definition of $\ptp$ and a gentle account of its basic properties, see \cite[Ch.~2]{Ryan_TP}.

If $\psi_1: E_1\to F_1$ and $\psi_2:E_2\to F_2$ are bounded linear maps between Banach spaces, we shall write $\psi_1\ptp \psi_2$ for the bounded linear map $E_1\ptp E_2\to F_1\ptp F_2$ that is defined by
\[ (\psi_1\ptp \psi_2)(x_1\tp x_2) \defeq \psi_1(x_1)\tp \psi_2(x_2)\quad\quad(x_1\in E_1, x_2\in E_2). \]
\end{notn}

\subsubsection*{Modules over a Banach algebra}
If $A$ is a Banach algebra then our definition of a left Banach $A$-module is the standard one: we require that the action of $A$ is continuous but do not assume that it is necessarily contractive. We shall assume the reader is familiar with the definition of left, right and two-sided Banach modules: for details see the introductory sections of \cite{Hel_HBTA}.

Throughout this article the phrase `$A$-module map' will be used to mean `map preserving $A$-module structure'. \emph{In particular, $A$-module maps are always linear}\/. (Alternative names for the same concept include `$A$-module morphism', or `$A$-module homo\-morphism; the terms seem to be used interchangeably in many accounts of ring theory, and we have merely chosen the shortest one.)


\begin{notn}
We fix notation for some familiar categories which will be referred to later.
$\Ban$ will denote the category whose objects are Banach spaces and whose morphisms are the continuous linear maps between Banach spaces.

If $A$ is a Banach algebra then we denote by $\LMod{A}$, $\RMod{A}$ the categories of left and right Banach $A$-modules respectively; in both cases the morphisms are taken to be the bounded left (respectively right) $A$-module maps. If $B$ is another Banach algebra then we let $\LRMod{A}{B}$ denote the category of Banach $A$-$B$-bimodules and $A$-$B$-bimodule maps.

If $A$ and $B$ are unital Banach algebras, then the corresponding categories of \emph{unit-linked} modules and module maps will be denoted by $\LunMod{A}$, $\RunMod{A}$ and $\LRunMod{A}{B}$ respectively.
\end{notn}

\end{subsection}

\begin{subsection}[Hochschild (co)homology]{Hochschild homology and cohomology for Banach algebras}\label{s:Hoch-coho}
There are several accounts of the basic definitions that we need: see
\cite{Hel_HBTA} for instance. However, we need some finer detail which carries over directly from the purely algebraic setting but seems not to be stated explicitly in the Banach algebraic setting.

We therefore briefly set out the relevant definitions, which also allows us to fix notation for what follows.

\begin{defn}
Let $A$ be a Banach algebra (not necessarily unital) and let $M$ be a Banach $A$-bimodule. For $n \geq 0$ we define
\begin{equation}
\begin{aligned}
\Ho{C}{n}(A,M) & \defeq M\ptp A^{\ptp n} \\
 \Co{C}{n}(A,M) & \defeq \{ \text{bounded, $n$-linear maps} \; \overbrace{A\times\ldots\times A}^n \to M \} \\
\end{aligned}
\end{equation}

For $0\leq i\leq n+1$ the \dt{face maps} $\face[n]{i}:\Ho{C}{n+1}(A,M)\to\Ho{C}{n}(A,M)$ are the contractive linear maps given by
\[ \face[n]{i}(x\tp a_1\tp\ldots\tp a_{n+1}) 
 = \left\{\begin{aligned}
      xa_1\tp a_2\tp \ldots\tp a_{n+1} & \quad\text{ if $i=0$} \\
      x\tp a_1\tp\ldots\tp a_i a_{i+1}\tp \ldots\tp a_{n+1} & \quad\text{ if $1\leq i \leq n$} \\
    a_{n+1}x\tp a_1 \tp \ldots \tp a_n & \quad\text{ if $i=n+1$}
 \end{aligned} \right. \]
and the \dt{Hochschild boundary operator} $\bdy_n:\Ho{C}{n+1}(A,M)\to\Ho{C}{n}(A,M)$ is given by
\[ \bdy_n = \sum_{j=0}^{n+1} (-1)^j \face[n]{j} \;. \]
With these definitions, the Banach spaces $\Ho{C}{n}(A,M)$ assemble into a chain complex
\[ \ldots \lTo^{\bdy_{n-1}} \Ho{C}{n}(A,M) \lTo^{\bdy_n} \Ho{C}{n+1}(A,M) \lTo^{\bdy_{n+1}} \ldots \]
called the \dt{Hochschild chain complex} of $(A,M)$.

Dually, the Banach spaces $\Co{C}{n}(A,M)$ assemble into a cochain complex
\[ \ldots \rTo^{\dif^{n-1}} \Co{C}{n}(A,M) \rTo^{\dif^n} \Co{C}{n+1}(A,M) \rTo^{\dif^{n+1}} \ldots \]
(the \dt{Hochschild cochain complex} of $(A,M)$), where the \dt{Hochschild coboundary operator} $\dif$ is given by
\[ \dif^n\psi(a_1,\ldots, a_{n+1}
 = \left\{\begin{aligned}
      & a_1\psi a_2,\ldots, a_{n+1}) \\
    + & \sum_{j=1}^n (-1)^j \psi(a_1, \ldots, a_ja_{j+1}, \ldots, a_{n+1})\\
    + & (-1)^{n+1} \psi(a_1,\ldots, a_n)a_{n+1}
 \end{aligned} \right.\]

We let
\begin{equation}
\begin{aligned}
\Ho{Z}{n}(A,M) & \defeq \Ker \bdy_{n-1} & \text{(the space of \dt{$n$-cycles})} \\
\Ho{B}{n}(A,M) & \defeq \Image \bdy_n & \text{(the space of \dt{$n$-boundaries})} \\
\Ho{H}{n}(A,M) & \defeq  \frac{\Ho{Z}{n}(A,M) }{\Ho{B}{n}(A,M) } & \text{(the \dt{$n$th Hochschild homology group}\index{Hochschild homology})} \\
\end{aligned}
\end{equation}
Similarly,
\begin{equation}
\begin{aligned}
\Co{Z}{n}(A,M) & \defeq \Ker \dif_n & \text{(the space of \dt{$n$-cocycles})} \\
\Co{B}{n}(A,M) & \defeq \Image \dif_{n-1} & \text{(the space of \dt{$n$-coboundaries})} \\
\Co{H}{n}(A,M) & \defeq  \frac{\Co{Z}{n}(A,M) }{\Co{B}{n}(A,M) } & \text{(the \dt{$n$th Hochschild cohomology group})} \\
\end{aligned}
\end{equation}
\end{defn}
\begin{remstar}
In the literature the spaces defined above are often referred to as the space of \emph{bounded} $n$-cycles, \emph{continuous} $n$-cocycles, etc. and the resulting homology and cohomology groups are then called the \dt{continuous} Hochschild homology and cohomology groups, respectively, of $(A,M)$. We have chosen largely to omit these adjectives as we never deal with the purely algebraic Hochschild cohomology of Banach algebras.

However, we shall on occasion refer to the purely algebraic theory: the corresponding spaces of chains and cochains on a given algebra will be denoted by $\Ho[\alg]{C}{*}$ and
 $\Co[\alg]{C}{*}$\/. For a condensed summary of the relevant definitions see \cite[Ch.~9]{Weibel}.
\end{remstar}

\end{subsection}

\begin{subsection}{Symmetric coefficients}\label{ss:symmetric}
For commutative Banach algebras it is rather natural to focus on those coefficient modules $M$ which are \dt{symmetric}, i.e.~such that $am=ma$ for all $a\in A$ and all $m\in M$\/.  In this context the following observation will prove useful, even if it seems rather trivial at first.

\begin{propn}\label{p:CBA_Hoch}
Let $A$ be a commutative Banach algebra and let $M$ be a \emph{symmetric} Banach $A$-bimodule. For each $n$, regard $\Ho{C}{n}(A,M)$ and $\Co{C}{n}(A,M)$ as left Banach $A$-modules, via the actions
\[ c\cdot (m\tp a_1\tp\ldots\tp a_n) \defeq (c\cdot m)\tp a_1\tp\ldots\tp a_n \quad\quad(m\in M ; c, a_1, \ldots, a_n\in A) \]
and
\[ (c\cdot T)(a_1,\ldots, a_n) \defeq c\cdot [T(a_1, \ldots, a_n)]\]
respectively. Then the boundary maps $\bdy_n:\Ho{C}{n+1}(A,M)\to \Ho{C}{n}(A,M)$ and the coboundary maps $\dif_n:\Co{C}{n}(A,M)\to\Co{C}{n+1}(A,M)$ are $A$-module maps.

In particular, the Hochschild chain complex
\[ \Ho{C}{0}(A,\cu{A})\lTo \Ho{C}{1}(A,\cu{A})\lTo\Ho{C}{2}(A,\cu{A}) \lTo \ldots\]
is a complex of Banach $A$-modules, and we have the following isometric isomorphisms of chain complexes:
\[ \begin{aligned} \Ho{C}{*}(A,M) & \miso M_R\ptp_{\cu{A}} \Ho{C}{*}(A,\cu{A}) \\
\Co{C}{*}(A,M) & \miso \lHom{(\cu{A})} \left(\Ho{C}{*}(A,\cu{A}), M_L\right)  \end{aligned} \]
where $M_L$ and $M_R$ are the one-sided modules obtained by restricting the action on $M$ to left and right actions respectively.
\end{propn}
The proposition is really just a statement about the boundary and coboundary operators, and its proof is immediate from their definition.


The idea to introduce this extra structure on the Hochschild chain complex is not at all original, but there seems to have been no systematic pursuit of this line of enquiry in the Banach-algebraic setting. One theme of this article is that for commutative Banach algebras, simplicial homology ought to control cohomology with symmetric coefficients: one may think of this as a kind of `universal coefficient theorem'.

In the purely algebraic setting this vague statement can be made into a precise result, which asserts that for any unital commutative algebra $\sA$ over a field and any symmetric $\sA$-bimodule ${\sf M}$, there is a spectral sequence
\begin{equation}\label{eq:ss-Hochsym}
\algExt^p_{\sA}\left(\Ho[\alg]{H}{q}(\sA, \sA),\, {\sf M} \right) \Rightarrow_p \Co[\alg]{H}{p+q}(\sA,{\sf M}) 
\end{equation}
which computes Hochschild cohomology in terms of simplicial homology of $\sA$ and the properties of $M$ as a \emph{one-sided} $\sA$-module. More background remarks can be found in \cite[\S3.2.1]{YC_PhD}.

\end{subsection}
\end{section}

\begin{section}{The Hodge decomposition of a commutative algebra}\label{s:BGSdefs}
The `Hodge decomposition' of the title gives a decomposition of the Hochschild homology and cohomology of a commutative algebra in characteristic zero. It was first introduced in Gerstenhaber and Schack's paper \cite{GS_Hodge}; for some of the history and context behind that paper, the reader is recommended to consult Gerstenhaber's excellent survey article~\cite{G_Barr}.

We shall follow the exposition in \cite[\S 9.4.3]{Weibel} which provides a terse guide. More details can be found in Loday's book \cite{Lod_CH}.

\begin{remstar}
This section consists mostly of standard material from commutative algebra, with the adjectives `Banach' or `bounded' inserted in the obvious places. However, there do not seem to be any explicit references for the Banach-algebraic case. We shall therefore endeavour to give precise statements, even when the proofs are trivial; the alternative approach would have led to tiresome repetition of the phrase `just as in the purely algebraic case, the reader may check that \ldots'. 
\end{remstar}

Let us start in the setting of $\Cplx$-algebras. Fix $n \in \Nat$: then for any $\Cplx$-vector space $V$ the permutation group $\Sym{n}$ acts on $V^{\tp n}$. This induces an action of the group algebra $\Rat\Sym{n}$ on the vector space $V^{\tp n}$: we shall identify elements of $\Rat\Sym{n}$ with the linear maps $V^{\tp n}\to V^{\tp n}$ that they induce. We recall also that if $\sB$ is a $\Cplx$-algebra and $\sf M$ a $\sB$-bimodule then $\Ho[\alg]{C}{*}(\sB,\sf M)$ denotes the corresponding Hochschild chain complex.

With this notation, we can now state the so-called `Hodge decomposition' of Gerstenhaber and Schack in a form convenient for us.
\begin{thm}[Hodge decomposition for commutative $\Cplx$-algebras]\label{t:BGS_alg}
Let $\sB$ be a commutative $\Cplx$-algebra. For each $n \geq 1$ there are pairwise orthogonal idempotents in $\Rat\Sym{n}$, denoted $\BGS[n]{1},\BGS[n]{2}, \ldots$, which satisfy
\begin{enumerate}
\renewcommand{\labelenumi}{{\rm(}\roman{enumi}{\rm)}}
\renewcommand{\theenumi}{{\rm(}\roman{enumi}{\rm)}}
\item $\BGS[n]{j}=0$ for all $j > n$;
\item $\sum_i\BGS[n]{i} = \id[\Rat\Sym{n}]$;
\end{enumerate}
\emph{and} are such that for each $i \in \Nat$, $\sid\tp \BGS[\blob]{i}$ acts a chain map on $\Ho[\alg]{C}{n}(\sB,\cu{\sB})$, i.e.~the diagram shown in Figure \ref{fig:BGSchainmap}
commutes for each $i, n\in\Nat$.
\end{thm}
\begin{figure}
\[ \begin{diagram}
 \Ho[\alg]{C}{n-1}(\sB,\cu{\sB}) & \lTo^{\bdy_{n-1}} & \Ho[\alg]{C}{n}(\sB,\cu{\sB}) \\
 \dTo^{\sid\tp\BGS[n-1]{i}} & & \dTo_{\sid\tp\BGS[n]{i}} \\
 \Ho[\alg]{C}{n-1}(\sB,\cu{\sB}) & \lTo_{\bdy_{n-1}} & \Ho[\alg]{C}{n}(\sB,\cu{\sB}) \\
\end{diagram} \]
\caption{Compatibility of idempotents with the boundary map}
\label{fig:BGSchainmap}
\end{figure}
\begin{proof}
See \cite[\S 9.4.3]{Weibel}.
\end{proof}

\begin{remstar}
 We have followed the notation from \cite{GS_Hodge}; what we have written as $\BGS[n]{i}$ is often denoted elsewhere in the literature by $e_n^{(i)}$.
\end{remstar}

The following is then obvious, and is stated for reference.
\begin{thm}[Hodge decomposition for commutative Banach algebras]\label{t:BGS_BAlg}
Let $A$ be a commutative Banach algebra. Each $\sid\ptp\BGS[n]{i}$ acts as a bounded linear projection on $\Ho{C}{n}(A,\cu{A})$; moreover, for fixed $i$ the family $\sid\ptp\BGS[*]{i}$ acts as a chain map on $\Ho{C}{*}(A,\cu{A})$.
\end{thm}
\begin{proof}
It is clear that $\sid\ptp\BGS[n]{i}$ acts boundedly on the Banach space $\Ho{C}{n}(A,\cu{A})$ -- and that the norm of the induced linear projection is bounded by some constant depending only on $i$ and $n$.

The remaining properties follow now by continuity, using Theorem~\ref{t:BGS_alg} and the density of the algebraic tensor product inside the projective tensor product.
\end{proof}

In the algebraic case we could have replaced $\cu{\sB}$ with any symmetric $\sB$-bimodule. The same is true in the Banach context.
\begin{coroll}\label{c:BGS_arb-coeff}
Let $A$ be a commutative Banach algebra and let $M$ be a symmetric Banach $A$-bimodule. Then for each $i$, $\sid_M\ptp\BGS[*]{i}$ is a bounded chain projection on $\Ho{C}{*}(A,M)$, and `pre-composition with $\BGS[*]{i}$' is a bounded chain projection on $\Co{C}{*}(A,M)$. 
\end{coroll}

\begin{proof}
Since $\BGS[n]{i}$ acts as a bounded linear projection on $A^{\ptp n}$, $\sid_M\ptp\BGS[n]{i}$ acts as a bounded linear projection on $M\ptp A^{\ptp n}=\Ho{C}{n}(A,M)$\/, and pre-composition with $\BGS[n]{i}$ acts as a bounded linear projection on $\Co{C}{n}(A,M)$. Therefore it only remains to show that these two maps are chain maps on the Hochschild chain and cochain complexes respectively.

This is essentially a trivial deduction from the case where $M=\cu{A}$. In more detail: recall (Proposition~\ref{p:CBA_Hoch}) that there are isomorphisms of Banach complexes
\begin{subequations}
\begin{align}
\label{eq:UCT_ho}\Ho{C}{*}(A,M) & \miso M\ptp_{\cu{A}} \Ho{C}{*}(A,\cu{A}) \\
\label{eq:UCT_co}\Co{C}{*}(A,M) & \miso \lHom{(\cu{A})} \left(\Ho{C}{*}(A,\cu{A}), M\right)  \;.
\end{align}
\end{subequations}
We have seen that for each $n$ and each $i$ there is a commuting diagram
\begin{equation}\label{diag:BGS_BAlg}
\begin{diagram}
\Ho{C}{n-1}(A,\cu{A}) & \lTo^{\bdy_{n-1}} & \Ho{C}{n}(A,\cu{A}) \\
 \dTo^{\sid\ptp\BGS[n-1]{i}} & & \dTo_{\sid\ptp\BGS[n]{i}} \\
\Ho{C}{n-1}(A,\cu{A}) & \lTo_{\bdy_{n-1}} & \Ho{C}{n}(A,\cu{A}) \\
\end{diagram}
\end{equation}
in which all arrows are continuous $\cu{A}$-module maps. Hence applying the functor $M\ptpR{\cu{A}}\blank$ and using Equation \eqref{eq:UCT_ho} yields a commuting diagram of Banach spaces:
\[ \begin{diagram}
\Ho{C}{n-1}(A,M) & \lTo^{\bdy_{n-1}} & \Ho{C}{n}(A,M) \\
 \dTo^{\sid\ptp\BGS[n-1]{i}} & & \dTo_{\sid\ptp\BGS[n]{i}} \\
\Ho{C}{n-1}(A,M) & \lTo_{\bdy_{n-1}} & \Ho{C}{n}(A,M) \\
\end{diagram} \]
as required. Similarly, applying the functor $\lHom{\cu{A}}(\blank, M)$ to Diagram \eqref{diag:BGS_BAlg} and using Equation \eqref{eq:UCT_co} gives a commuting diagram
\[ \begin{diagram}
\Co{C}{n-1}(A,M) & \rTo^{\dif^{n-1}} & \Co{C}{n}(A,M) \\
 \dTo^{\coBGS[n-1]{i}} & & \dTo_{\coBGS[n]{i}} \\
\Co{C}{n-1}(A,M) & \rTo_{\dif^{n-1}} & \Co{C}{n}(A,M) \\
\end{diagram} \]
and the proof is complete.
\end{proof}

\begin{defn}
Let $A$ be a commutative Banach algebra and $M$ a symmetric Banach $A$-bimodule. For $n \in \Nat$ and $i=1, \ldots, n$ we follow the notation of \cite{GS_Hodge} and write
\[ \begin{aligned} \Ho{C}{i,n-i}(A,M) & \defeq (\sid\ptp\BGS[n]{i})\Ho{C}{n}(A,M) \\
 \Co{C}{i,n-i}(A,M) & \defeq \coBGS[n]{i}\Co{C}{n}(A,M) \end{aligned} \]
where $\coBGS[n]{i}$ is defined to be `pre-composition with $\BGS[n]{i}$'.

Given a chain or cochain, we shall sometimes say that it is of \dt{BGS type $(i,n-i)$} if it lies in the corresponding summand $\Ho{C}{i,n-i}$ or $\Co{C}{i,n-i}$. We shall also sometimes refer to the projections $\sid\ptp\BGS[n]{i}$ and $\coBGS[n]{i}$ as the \dt{BGS projections} on homology and cohomology respectively. (This terminology comes from the survey article \cite{G_Barr}; the acronym `BGS' is for Barr-Gerstenhaber-Schack.)

Since $\sid\ptp\BGS{i}$ is a chain projection for each $i$, we have a decomposition of the chain complex $\Ho{C}{*}(A,M)$ into orthogonal summands; dually, the chain projections $(\coBGS{i})_{i\geq 1}$ yield a decomposition of the cochain complex $\Co{C}{*}(A,M)$ into orthogonal summands.
For both homology and cohomology the decomposition has $n$ summands in degree~$n$\/.

This is the so-called \dt{Hodge decomposition} of Hochschild (co)homology (the origin of the name is explained in \cite{G_Barr}\/).
\end{defn}

\begin{remstar}
Note that in passing, the proof of Corollary~\ref{c:BGS_arb-coeff}, shows that there are chain isomorphisms
\begin{subequations}
\begin{align}
\label{eq:UCT_BGSho}\Ho{C}{i,*}(A,M) & \miso M\ptp_{\cu{A}} \Ho{C}{i,*}(A,\cu{A}) \\
\label{eq:UCT_BGSco}\Co{C}{i,*}(A,M) & \miso \lHom{(\cu{A})} \left(\Ho{C}{i,*}(A,\cu{A}), M\right)  \;.
\end{align}
\end{subequations}
for any $i$ (the case $i=1$ will be used in some later calculations).
\end{remstar}

\subsection*{Harrison and Lie cohomology}
At first glance the various subscripts and superscripts may cloud the picture unnecessarily. It is therefore useful to have in mind a schematic diagram such as Figure~\ref{fig:BGS_schematic} (for cohomology).
\begin{figure}
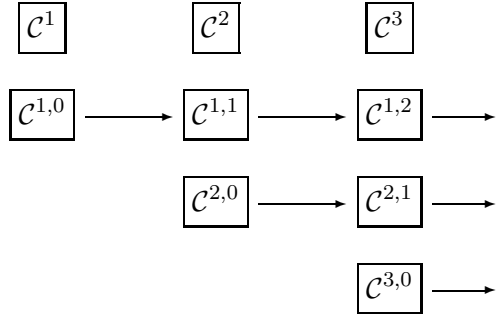

\begin{diagram}[tight, width=2em]
 \boxed{\Co{C}{1}} &  & \boxed{\Co{C}{2}} & & \boxed{\Co{C}{3}} &  & \\
 \boxed{\Co{C}{1,0}} & \rTo & \boxed{\Co{C}{1,1}} & \rTo & \boxed{\Co{C}{1,2}} & \rTo  \\
    &   & \boxed{\Co{C}{2,0}}    & \rTo & \boxed{\Co{C}{2,1}} & \rTo  \\
    &   &    &   & \boxed{\Co{C}{3,0}} & \rTo  \\
\end{diagram}
\caption{Hodge decomposition for cohomology}
\label{fig:BGS_schematic}
\end{figure}
In this schematic, there are two distinguished parts of the Hodge decomposition: we may consider the bottom box in each column, or the top one. These 
components of the decomposition warrant names of their own.

We first consider the summands of BGS type $(n,n)$. While explicit formulas for the idempotents $\BGS[n]{i}$ are hard to work with in general, the idempotents $\BGS[n]{n}$ turn out to be familiar and tractable. We state the following without proof.

\begin{thm}\label{t:BGSalt}
For $n\geq 1$ we have
\begin{equation}
\BGS[n]{n}= \frac{1}{n!} \sum_{\sig \in \Sym{n}} (-1)^\sig \sig \in \Rat\Sym{n} 
\end{equation}
Thus
the summands $\Ho{C}{n,0}(A,M)$ and $\Co{C}{n,0}(A,M)$ turn out to be the spaces of alternating chains and cochains.
\end{thm}
For more details see \cite[Propn 2.1]{Barr_Harr} or \cite[Lemma 9.4.9]{Weibel}.

In light of this fact we adopt the following terminology.
\begin{defn}
The \dt{Lie component of degree $n$} is the space $\Co{C}{n,0}(A,M)$ of continuous, alternating $n$-cochains from $A$ to~$M$. 
\end{defn}

\begin{rem}
The name `Lie component' follows the discussion in \cite[Thms~5.9, 5.10]{GS_Hodge} which loosely says that for a commutative $\Rat$-algebra $\sB$ and symmetric bimodule~$\sf M$, $\Co[\alg]{H}{n,0}(\sB,\sf M)$ is isomorphic to the Lie algebra cohomology of the pair $(\sB, \sf M)$.
\end{rem}

We shall not discuss the Lie component in this article, save to point out that it was re\-discovered (under a different name) in Johnson's paper~\cite{BEJ_high}. The central notion of that paper was a definition of \dt{$n$-dimensional weak amenability}; in the language adopted here, a commutative Banach algebra $A$ is $k$-dimensionally weakly amenable if $\Co{H}{n,0}(A,M)=0$ for all $n \geq k$.

Instead, we shall focus in the rest of this paper on the other extreme, namely the spaces $\Co{H}{1,n-1}(A,M)$. These are known as the \dt{Harrison cohomology} groups of $(A,M)$, and will be discussed in more detail in the next section.

\end{section}

\begin{section}{Harrison homology and (co)homology}\label{s:Harrison}
\begin{defn}
The complex $\Co{C}{1,*}$ is called the \dt{Harrison summand} of the Hochschild chain complex, and its cohomology is called \dt{Harrison cohomology}. Dually, the complex $\Ho{C}{1,*}$ is the Harrison summand of the Hochschild chain complex, and its homology is called \dt{Harrison homology}\/.
\end{defn}

\begin{notn}
From here on, when focusing on the Harrison summand and not on the Hodge decomposition in general, we shall adopt the alternative notation $\HarC{C}{n}\defeq\Co{C}{1,n-1}$, $\HarH{C}{n}\defeq\Ho{C}{1,n-1}$,~etc.
\end{notn}

\begin{remstar}
Since $\BGS[2]{1}+\BGS[2]{2}=\sid$, we see that in degree~$2$ the Hodge decomposition coincides with the decomposition of (co)homology into symmetric and anti-symmetric summands (with the symmetric part being the Harrison summand).
\end{remstar}

In the purely algebraic setting, the complex of Harrison cochains was introduced and studied some 20 years before the general `Hodge decomposition' was formulated by Gerstenhaber and Shack. For more historical background we recommend the remarks in \cite{GS_Hodge} and the account in \cite{G_Barr}.

To give some idea of what we are aiming for in our main result (Theorem~\ref{t:Harr_of_Ak} below) we briefly discuss some aspects of Harrison cohomology in the purely algebraic setting. In Harrison's original 1962 paper \cite{HarrCo_62}, a \Kunneth-type theorem is stated:
\begin{equation}\label{eq:algHarr-of-tp}
\HarC[\alg]{H}{n}(\sA\tp\sB, {\sf M}) \iso \HarC[\alg]{H}{n}(\sA,{\sf M}) \oplus \HarC[\alg]{H}{n}(\sB, {\sf M}) \end{equation}
Harrison only gives the proof for degrees 1, 2 and 3: his proof involves explicit manipulation of cochains and ought to translate to the Banach-algebraic setting. However, the only proofs in the literature \emph{for general $n$} seem to rely on spectral-sequence arguments and the fact that the Harrison cohomology of a polynomial algebra in arbitrarily many variables vanishes in degrees $2$ and above (see Theorem~\ref{t:polyring_HH} below for more details). Since we do not know if the corresponding statement is true for the Banach algebra $\lp{1}(\Z_+^\infty)$, we have been unable to establish the Banach-algebraic version of \eqref{eq:algHarr-of-tp} in full generality: Theorem~\ref{t:Harr_of_Ak} provides evidence that \emph{some} Banach-algebraic version ought to be true.

\subsection*{Long exact sequences}
The Hodge decomposition of Hochschild (co)homology respects the usual long exact sequences associated to certain short exact sequences of coefficient modules.
 We shall only need this for the special case of Harrison (co)homology: the precise formulation is as follows.

\begin{lemma}[Long exact sequences for Harrison (co)homology]\label{l:Harr_LES}
Let $A$ be a commutative Banach algebra, and let $L \to M \to N$ be a short exact sequence of symmetric Banach $A$-bimodules which is split exact in $\Ban$. Then there are long exact sequences of Harrison homology
\[  0\leftarrow \HarH{H}{1}(A,N) \leftarrow \HarH{H}{1}(A,M) \leftarrow \HarH{H}{1}(A,L)\leftarrow \HarH{H}{2}(A,N) \leftarrow \ldots \]
and Harrison cohomology
\[ 0\rightarrow \HarC{H}{1}(A,L) \rightarrow \HarC{H}{1}(A,M) \rightarrow \HarC{H}{1}(A,N) \rightarrow \HarC{H}{2}(A,L) \rightarrow \ldots \]
\end{lemma}
\begin{proof}
We shall give the proof for Harrison homology and omit that for cohomology since the proof technique is identical.

Since $L\to M \to N$ is split in $\Ban$, so is the induced short exact sequence
\[ L\ptp A^{\ptp n} \to M\ptp A^{\ptp n} \to N \ptp A^{\ptp n} \]
and it remains split if we apply the BGS idempotent  $\sid\ptp \BGS[n]{1}$ to each term in the sequence. But by the definition of Harrison homology the resulting split exact sequence of Banach spaces is just
\[ \HarH{C}{n}(A,L)\to \HarH{C}{n}(A,M)\to \HarH{C}{n}(A,N) \]
Thus we have a short exact sequence of complexes
\[ \HarH{C}{*}(A,L)\to \HarH{C}{*}(A,M)\to \HarH{C}{*}(A,N) \]
and the standard diagram chase allows us to construct from this a long exact sequence of homology.

Furthermore, in the portion of the long exact sequence which goes
\[ 0\leftarrow \HarH{H}{0}(A,N) \leftarrow \HarH{H}{0}(A,M) \lTo^{\iota_0} \HarH{H}{0}(A,L)\lTo^{\rm conn} \HarH{H}{1}(A,N) \leftarrow \ldots \]
we observe that $\HarH{H}{0}(A,X)=\Ho{H}{0}(A,X)=X$ for any symmetric $A$-bimodule $X$. Hence $\iota_0$ is just the inclusion of $L$ into $M$ and is in particular injective; we deduce that the connecting map ${\rm conn}: \HarH{H}{1}(A,N) \to \HarH{H}{0}(A,L)$ is zero, and so our long exact sequence starts
\[   0\lTo^{\rm conn} \HarH{H}{1}(A,N) \leftarrow \HarH{H}{1}(A,M) \leftarrow \ldots \]
as claimed.
\end{proof}
\begin{remstar}
It is clear that similar long exact sequences exist for each summand $\Co{C}{i,*}$ in the Hodge decomposition of cohomology, and for each summand $\Ho{C}{i,*}$ in the Hodge decomposition of homology. We omit the details since they will not be needed in what follows.
\end{remstar}

\subsection*{Harrison (co)homology as a derived functor}
The following computations are motivated by the spectral sequence discussed at the end of Section~\ref{ss:symmetric}.

\begin{propn}\label{p:Harr-as-Ext}
Let $B$ be a commutative Banach algebra such that the chain complex
\begin{equation}\label{eq:Harr-exact}
 \Ho{C}{1}(B,\cu{B})\lTo^{\bdy_1} \HarH{C}{2}(B,\cu{B}) \lTo^{\bdy_2}\HarH{C}{3}(B,\cu{B})\lTo^{\bdy_3}\ldots \end{equation}
is split exact in $\Ban$. Then $\Ho{H}{1}(B,\cu{B})$ is a left Banach $B$-module, which is unit-linked if $B$ is unital.

Moreover, for each $n\geq 1$ and any symmetric Banach $B$-bimodule~$X$\/, we have isomorphisms of seminormed spaces
\[ \begin{aligned}
\HarC{H}{n}(B,X) & \iso \Ext^{n-1}_B\left( \Ho{H}{1}(B,\cu{B}), X_L\right) \\ 
\HarH{H}{n}(B,X) & \iso \Tor_{n-1}^B\left(X_R, \Ho{H}{1}(B,\cu{B})\right)
\end{aligned}\]
where $X_L$ and $X_R$ denote the $B$-modules obtained by restricting the $2$-sided action on $B$ to a left and right action respectively.
\end{propn}
\begin{proof}
Recall from Proposition~\ref{p:CBA_Hoch}
that
\[  \Ho{C}{1}(B,\cu{B})\lTo^{\bdy_1} \HarH{C}{2}(B,\cu{B}) \lTo^{\bdy_2}\HarH{C}{3}(B,\cu{B})\lTo^{\bdy_3}\ldots \]
is a complex in $\LunMod{B}$.

The hypothesis \eqref{eq:Harr-exact} says that there exist bounded linear maps
\[ \sig_n:\HarH{C}{n}(B,\cu{B}) \to \HarH{C}{n+1}(B,\cu{B}) \]
such that
\[ \sig_n\bdy_n+\bdy_{n+1}\sig_{n+1}=\sid \quad\quad\text{ for $n=1,2,\ldots$}\;.\]
In particular $\sig_1\bdy_1=\sid-\bdy_2\sig_2$; hence $\bdy_1\sig_1\bdy_1=\bdy_1(\sid-\bdy_2\sig_2)=\bdy_1$ and this implies that $\bdy_1$ has closed range. Thus $\Ho{H}{1}(B,\cu{B})=\Coker(\bdy_1)$ is the quotient of a Banach $B$-module by a closed submodule, and is therefore itself a Banach $B$-module. If $B$ is unital then $\Ho{H}{1}(B,B)$ is unit-linked, since $\Ho{C}{1}(B,B)$ is.

Each $\Ho{C}{n}(B,\cu{B})$ is $B$-projective as a Banach $B$-module (since $\cu{B}$ is); therefore, since the BGS projections are $B$-module maps, each $\HarH{C}{n}(B,\cu{B})$ is a $B$-module summand of a $B$-projective module and is thus $B$-projective. Hence by the hypothesis \eqref{eq:Harr-exact} the complex
\[ 0\leftarrow \Ho{H}{1}(B,\cu{B})\lTo^q \Ho{C}{1}(B,\cu{B})\lTo^{\bdy_1} \HarH{C}{2}(B,\cu{B}) \lTo^{\bdy_2}\ldots \]
is an admissible $B$-projective resolution of $\Ho{H}{1}(B,\cu{B})$, and by the definitions of $\Tor$ and $\Ext$ we have
\[ \Ext_B^{n-1}\left[ \Ho{H}{1}(B,\cu{B}), X_L \right] \iso H^n\left[ \lHom{B}\left(\HarH{C}{*}(B,\cu{B}), X_L\right)\right] \]
and
\[ \Tor^B_{n-1}\left[ X_R, \Ho{H}{1}(B,\cu{B})\right] \iso H_n\left[ X_R\ptpR{B}\HarH{C}{*}(B,\cu{B})\right] \]
for every $n\geq 1$.
To finish, we recall (see Equation~\eqref{eq:UCT_BGSco}) that the cochain complex $\lHom{B}\left(\HarH{C}{*}(B,\cu{B}), X_L\right)$
is isomorphic to $\HarC{C}{*}(B,X)$, and that the chain complex $X_R\ptpR{B}\HarH{C}{*}(B,\cu{B})$ is isomorphic to $\HarH{C}{*}(B,X)$.
\end{proof}

\end{section} 

\begin{section}{A `baby \Kunneth\ formula'}\label{s:babyKunn}
The \Kunneth\ formula of \cite{GouLyWh} is applied in that article to calculate the simplicial homology groups of $\lp{1}(\Z_+^k)$ up to isomorphism of seminormed spaces; in particular one sees that $\Ho{H}{n}(\lp{1}(\Z_+^k),\lp{1}(\Z_+^k))$ is Banach for all $n$ and all $k$. For later reference, we would like to determine the first simplicial homology group of $\lp{1}(\Z_+^k)$ up to \emph{isomorphism of Banach $\lp{1}(\Z_+^k)$-modules}.

It should be possible, by chasing the relevant maps through the proofs in \cite{GouLyWh}, to show that the Banach-space isomorphism calculated there is in fact an $\lp{1}(\Z_+^k)$-module map. However, we have chosen a more abstract approach: for each unital commutative Banach algebra $A$ we construct a natural seminormed space $\qKahl{A}$ which is also an $A$-module; we show that $\qKahl{A}$ may be identified as a seminormed space and as an $A$-module with $\Ho{H}{1}(A,A)$\/; and we then give a decomposition theorem for $\qKahl{A\ptp B}$ whenever $A$ and $B$ are unital commutative Banach algebras. This approach is slightly more general than that in \cite{GouLyWh} for first homology groups, in that we do not {\it a priori} assume that either $\qKahl{A}$ or $\qKahl{B}$ is Hausdorff.


\subsection{Notation and other preliminaries}
Let $A$ be a unital commutative Banach algebra. Let $I_A$ denote the kernel of the product map $A\ptp A \to A$, equipped with the $A$-bimodule structure it inherits from $A\ptp A$.

We let $\sig_A$ denote the projection from $A\ptp A$ onto $I_A$ which is defined by
\[ \sig_A(x\tp y) = x\tp y - xy\tp \id[A] \]
and note that $\ker(\sig_A)=A\ptp \Cplx\id[A]$. Note also that $\sig_A$ is a left $A$-module map.


Let $\tau_A: A\ptp A \ptp A \to I_A$ be the bounded linear map defined by
\[\begin{aligned}
 \tau_A(x\tp y\tp a) & = \sig_A(x\tp y)\cdot a -a\cdot\sig_A(x\tp y) \\
 & = x\tp ya - xy\tp a - ax\tp y + axy\tp \id[A]
 \end{aligned}\]
Although $\Image(\tau_A)$ need not be closed in $I_A$, it is always a left $A$-submodule of $I_A$\/ (since $A$ is commutative). Hence the quotient space
\[ \qKahl{A}\defeq I_A/\Image(\tau_A) \]
inherits the structure of a left $A$-module.

Being the quotient of a Banach space by a subspace, $\qKahl{A}$ can be equipped with the quotient seminorm, and so we can meaningfully discuss bounded linear maps to and from it (see the remarks at the start of Section~\ref{ss:genprelim}).

The following result is somehow implicit in the setup of \cite{Run_Kahl}, but the precise formulation here is new as far as I know. It is a straightforward if fiddly modification of standard ideas from commutative algebra (see \cite[9.2.4]{Weibel} for instance).
\begin{propn}\label{p:HH1_is_qKahl}
There is an isomorphism of seminormed spaces $\qKahl{A}\to \Ho{H}{1}(A,A)$, which is also an isomorphism of left $A$-modules.
\end{propn}
Note that in particular $\qKahl{A}$ is a Banach space if and only if $\Ho{H}{1}(A,A)$ is.

\begin{proof}
Let $\bdy_1^A: \Ho{C}{2}(A,A) \to \Ho{C}{1}(A,A)$ be the Hochschild boundary map, given by the formula
\[ \bdy_1^A(x\tp a_1\tp a_2) = xa_1\tp a_2 - x\tp a_1a_2 + a_2x\tp a_1 \;.\]
Direct calculation yields the useful formula
\begin{equation}\label{eq:BRUNO}
\sig_A\bdy_1^A= - \tau_A\/.
\end{equation}
In particular the composite map
\[ \Ho{C}{1}(A,A) =A\ptp A \rTo^{\sig_A} I_A \rTo I_A/\Image(\tau_A) \]
vanishes on $\Image(\bdy_1^A)=\Ho{B}{1}(A,A)$, hence descends to a well-defined and bounded linear $A$-module map $\wtild{\sig_A}$ as shown in Figure~\ref{fig:descent} below.
\begin{figure}
\begin{diagram}
\Ho{C}{1}(A,A)  & \rTo^{\sig_A} & I_A \\
\dTo^{q} & & \dTo \\
\Ho{H}{1}(A,A) & \rDots_{\wtild{\sig_A}} & \qKahl{A} \\
\end{diagram}
\caption{Inducing a map between quotient spaces}
\label{fig:descent}
\end{figure}
It now suffices to construct a bounded linear 2-sided inverse to $\wtild{\sig_A}$, which we do as follows. Let $J: I_A \to A\ptp A =\Ho{C}{1}(A,A)$ be the inclusion map: then $\sig_AJ=\sid$. Moreover, for any $x,y,a \in A$
\[ \begin{aligned} J\tau_A(x\tp y\tp a) 
& = x\tp ay - xy\tp a - ax\tp y + axy\tp \id[A] \\
& = - \bdy_1^A(x\tp y\tp a) + \bdy_1^A(axy\tp\id[A]\tp\id[A]) \\
\end{aligned} \]
and so $qJ$ vanishes on $\Image(\tau)$, inducing a bounded linear map $\wtild{J}: \qKahl{A} \to \Ho{H}{1}(A,A)$. Since $\sig_A J=\sid$, $\wtild{\sig_A}\wtild{J}$ is the identity map, and it remains only to show that ${ \sid-J\sig_A }$ takes values in $\ker(q)=\Image(\bdy_1)$. But this is immediate, since
\[ (\sid-J\sig_A) (x\tp y) = xy\tp \id[A] = \bdy_1^A(xy\tp\id[A]\tp\id[A]) \]
for all $x,y\in A$\/.
\end{proof}

\subsection*{The main formula}
Let $A$ and $B$ be unital commutative Banach algebras; then their projective tensor product $A\ptp B$ is also a unital commutative Banach algebra, which we denote by $C$.

\begin{thm}[Differentials of tensor products]\label{t:qKahl_of_tp}
There exist mutually inverse, bounded linear $C$-module maps
\begin{equation}\label{eq:qKahl_of_tp}
\qKahl{C} \RL{\Ex}{\Ass} \frac{I_A\ptp B}{\Image(\tau_A\ptp\sid_B)} \oplus\frac{A\ptp I_B}{\Image(\sid_A\ptp\tau_B)}
\end{equation}
\end{thm}

\begin{coroll}\label{c:BARRY}
Suppose furthermore that the underlying Banach spaces of $A$ and $B$ are isomorphic to $\Lp{1}$-spaces, and that both $\tau_A$ and $\tau_B$ have closed range.
Then $\tau_C$ has closed range and we have an isomorphism of Banach $C$-modules
\[ \qKahl{C}\iso \qKahl{A}\ptp B \oplus A\ptp\qKahl{B} \;. \]
\end{coroll}
\begin{proof}[Proof of Corollary~\ref{c:BARRY}]
Write $V_A$ and $V_B$ for $\Image\tau_A$ and $\Image\tau_B$ respectively: since these are both closed subspaces, $\qKahl{A}= I_A/V_A$ and $\qKahl{B}=I_B/V_B$ are both Banach spaces.

Since $A$ is an $\Lp{1}$-space, standard Banach space theory tells us that the functor $A\ptp(\blank)$ sends short exact sequences of Banach spaces to short exact sequences.
(See, for instance, Propos\-ition~3.10 and the remark just after it in \cite{DefFlor}.)
 In particular, $A\ptp V_B$ is a closed subspace of $A\ptp I_B$ and
\[ A\ptp \qKahl{B} = A\ptp (I_B / V_B ) \iso \frac{A\ptp I_B}{\Image(\sid_A\ptp\tau_B)}\]
 By symmetry we also have $B\ptp\qKahl{A}\iso (I_A\ptp B)/\Image(\tau_A\ptp\sid_B)$, and so Equation \eqref{eq:qKahl_of_tp} simplifies to give
\[\qKahl{C}\iso \qKahl{A}\ptp B \oplus A\ptp\qKahl{B}  \]
(this shows in passing that $\tau_C$ has closed range).
\end{proof}

We may extend Corollary~\ref{c:BARRY} to $k$-fold tensor products of unital commutative Banach algebras, by the obvious induction on $k$\/.
\begin{coroll}\label{c:baby_Kunn}
Let $A_1, \ldots, A_k$ be unital, commutative Banach algebras, each of whose underlying Banach spaces is an $\Lp{1}$-space, and let $\fA = \widehat{\bigotimes}_{i=1}^k A_i$. Suppose that $\Ho{H}{1}(A_i,A_i)$ is a Banach space for each $i$. Then $\Ho{H}{1}(\fA,\fA)$ is Banach, and there is an isomorphism of Banach $\fA$-modules
\[ \Ho{H}{1}(\fA,\fA) \iso \bigoplus_{i=1}^k A_1\ptp \ldots \ptp \Ho{H}{1}(A_i,A_i)\ptp\ldots \ptp A_k \]
\end{coroll}

\subsection*{The proof of Theorem~\ref{t:qKahl_of_tp}}
What follows is simple but involves rather tedious manipulations. We shall construct suitable Banach $C$-module maps
\[ \preEx: I_C \to I_A\ptp B \oplus A\ptp I_B\quad\text{ and }\quad  \preAss:  I_A\ptp B \oplus A\ptp I_B \to I_C \]
which descend to mutually inverse maps at the level of quotient spaces. We proceed in three steps

\subsubsection*{Step 1: the definition of $\preEx$}
Let
\begin{equation}\label{eq:dfn_preEx}
 \preEx\left(\sig_C(a\tp b\tp x\tp y)\right) \defeq  \cvct{\sig_A(a \tp x)\tp by}{ax\tp\sig_B(b\tp y)}
 \quad\quad(a,x \in A; b,y \in B)
 \end{equation}
(this is well-defined, since $\ker(\sig_C)=C\ptp \Cplx\id[C] = A\ptp B\ptp\Cplx(\id[A]\tp\id[B])$ and the right-hand side of equation \eqref{eq:dfn_preEx} vanishes if $x\in\Cplx\id[A]$ and $y\in\Cplx\id[B]$). One easily checks that $\preEx$ is a Banach $C$-module map.

We must show that
there is a well-defined, bounded linear $C$-module map $\Ex$ which makes the following diagram commute:
\[ \begin{diagram}
I_C  & & \rTo^{\preEx} & & I_A\ptp B \oplus A\ptp I_B \\
\dTo &  & & & \dTo \\
\Coker(\tau_C) & & \rDots_{\Ex} & & \Coker(\tau_A\ptp\sid_B)\oplus\Coker(\sid_A\ptp\tau_B) \\
\end{diagram} \]
By standard diagram chasing, it suffices to show that 
\[ \Image(\preEx \circ\tau_C) \subseteq\Image((\tau_A\ptp\sid_B, \sid_A\ptp \tau_B)) \quad;\]
this inclusion in turn follows from the following claim:

\smallskip
\noindent{\bf Claim \#1.} There exists a bounded linear map $\tht$ making the following diagram commute:
\[ \begin{diagram}
C^{\ptp 3} & & \rTo^{\tau_C} & & I_C \\
 \dDots^{\tht} & &   & & \dTo_{\preEx} \\
A^{\ptp 3}\ptp B\oplus A\ptp B^{\ptp 3} & & \rTo_{(\tau_A\ptp\sid_B, \sid_A\ptp \tau_B)} & & I_A\ptp B\oplus A\ptp I_B \\
\end{diagram} \]
(For if we assume the claim holds, then
\[ \Image(\preEx \circ\tau_C) =\Image((\tau_A\ptp\sid_B, \sid_A\ptp \tau_B)\circ\tht) \subseteq\Image((\tau_A\ptp\sid_B, \sid_A\ptp \tau_B)) \]
as required.)

\begin{proof}[Proof of Claim \#1]
Let $x_1,x_2,a \in A$ and $y_1,y_2,b \in B$. Since $\tau=-\sig\bdy_1$ (see \eqref{eq:BRUNO} above) we have
\[ \begin{aligned}
 & \preEx\tau_C(x_1\tp y_1\tp x_2\tp y_2\tp a\tp b) \\
 & = - \preEx\sig_C\bdy_1^C(x_1\tp y_1\tp x_2\tp y_2\tp a\tp b) \\
 & = - \preEx\sig_C(x_1x_2\tp y_1y_2\tp a\tp b - x_1\tp y_1\tp x_2a\tp y_2b
	+ ax_1\tp by_1 \tp x_2\tp y_2) \\
 & =
	- \cvct{\sig_A(x_1x_2\tp a)\tp y_1y_2b}{x_1x_2a\tp \sig_B(y_1y_2\tp b}
	+ \cvct{\sig_A(x_1\tp x_2a)\tp y_1y_2b}{x_1x_2a\tp\sig_B(y_1\tp y_2\tp b)}
	- \cvct{\sig_A(ax_1\tp x_2)\tp by_1y_2}{ax_1x_2\tp\sig_b(by_1\tp y_2)}
 \\ 
 & = - \cvct{\sig_A\bdy_1^A(x_1\tp x_2\tp a) \tp y_1y_2b}{x_1x_2a\tp\sig_B\bdy_1^B(y_1\tp y_2\tp b)} \\
 & = \cvct{\tau_A(x_1\tp x_2\tp a)\tp by_1y_2 }{ax_1x_2 \tp \tau_B(y_1\tp y_2\tp b)} \;.\\
\end{aligned} \]
We therefore define $\tht$ by the formula
\[ \tht(x_1\tp y_1\tp x_2\tp y_2\tp a\tp b) \defeq \cvct{(x_1\tp x_2\tp a)\tp by_1y_2 }{ax_1x_2\tp(y_1\tp y_2\tp b)} \]
and observe that $\tht$ is bounded linear; by linearity and continuity the calculation above implies that $\preEx\tau_C = \tht(\tau_A\ptp\sid_B, \sid_A\ptp\tau_B)$ as claimed.
\end{proof}

\subsubsection*{Step 2: the definition of $\preAss$}
It is convenient to introduce auxiliary maps $\preAss_A:I_A\ptp B \to I_C$ and $\preAss_B:A\ptp I_B \to I_C$\/, defined by
\[ \begin{aligned}
\preAss_A(\sig_A(u\tp x)\tp b) & = \sig_C(u\tp b\tp x\tp\id[B]) &\quad\parbox{16em}{(well-defined, since the right-hand side vanishes if $x\in\Cplx\id[A]$)} \\
\preAss_B(a\tp\sig_B(v\tp y))  & = \sig_C(a\tp v\tp\id[A]\tp y) &\quad\parbox{16em}{(well-defined, since the right-hand side vanishes if $y\in\Cplx\id[B]$)} \\
 \end{aligned} \]
One checks easily that $\preAss_A$ and $\preAss_B$ are Banach $C$-module maps. Hence their direct sum
\[ \preAss\defeq \Ass_A\oplus\Ass_B : \cvct{I_A\ptp B}{A\ptp I_B} \to I_C\]
 is also a Banach $C$-module map.

We must show that
there is a well-defined, bounded linear $C$-module map $\Ass$ which makes the following diagram commute:
\[ \begin{diagram}
I_A\ptp B \oplus A\ptp I_B   & & \rTo^{\preAss} & & I_C \\
\dTo & &  & & \dTo \\
\Coker(\tau_A\ptp\sid_B)\oplus\Coker(\sid_A\ptp\tau_B) & & \rDots_{\Ass} & &  \Coker(\tau_C)\\
\end{diagram} \]
By standard diagram chasing, it suffices to show that 
\[ \Image(\preAss\circ(\tau_A\ptp\sid_B, \sid_A\ptp \tau_B))\subseteq\Image(\tau_C)  \quad;\]
this inclusion in turn follows from the following claim:

\smallskip
\noindent{\bf Claim \#2.} There exists a bounded linear map $\gm$ making the following diagram commute:
\[ \begin{diagram}
A^{\ptp 3}\ptp B\oplus A\ptp B^{\ptp 3} & & \rTo^{(\tau_A\ptp\sid_B, \sid_A\ptp \tau_B)} & & I_A\ptp B\oplus A\ptp I_B \\
 \dDots^{\gm} & &   & & \dTo_{\preEx} \\
C^{\ptp 3} & & \rTo_{\tau_C} & & I_C \\
\end{diagram} \]
(For if we assume the claim holds, then
\[ \Image(\preAss\circ(\tau_A\ptp\sid_B, \sid_A\ptp \tau_B)) = \Image(\tau_C\circ\gm) \subseteq \Image(\tau_C) \]
as required.)

\begin{proof}[Proof of Claim \#2]
Let $x_1,x_2,u \in A$ and $b \in B$. Since $\sig\bdy_1=-\tau$\/, we have
\[ \begin{aligned}
 & \preAss_A(\tau_A\ptp\sid_B)(x_1\tp x_2\tp u\tp b) \\
 & = - \preAss_A\left( \sig_A\bdy_1^A(x_1\tp x_2\tp a)\tp b\right) \\
 & = - \preAss_A\left(
	\sig_A(x_1x_2\tp u - x_1\tp x_2u + ux_1\tp x_2)\tp b
	\right) \\
 & = - \sig_C( x_1x_2\tp b\tp u\tp\id[B] - x_1\tp b\tp x_2u\tp\id[B] + ux\tp b\tp x_2\tp\id[B]  ) \\
 & = -\sig_C\bdy_1^C ( x_1\tp b\tp x_2\tp\id[B]\tp u\tp\id[B]) \\
 & = \tau_C(x_1\tp b\tp x_2\tp \id[B]\tp u\tp \id[B])  \quad;\\
\end{aligned}\]
and by symmetry, if $a\in A$ and $y_1,y_2,v \in B$, we have
\[ \preAss_B(\sid_A\ptp\tau_B)(a\tp y_1\tp y_2\tp v) = \tau_C(a\tp y_1\tp\id[A]\tp y_2\tp\id[A]\tp v) \;.\]
We therefore define $\gm$ by the formula
\[ \gm\cvct{(x_1\tp x_2\tp u)\tp b}{a\tp (y_1\tp y_2\tp v)} \defeq x_1\tp b\tp x_2\tp \id[B]\tp u\tp \id[B] + a\tp y_1\tp\id[A]\tp y_2\tp\id[A]\tp v\]
and observe that $\gm$ is bounded linear; by linearity and continuity the calculations above imply that
\[ \preAss_A\left( \tau_A\ptp\sid_B, \sid_A\ptp\tau_B\right) = \tau_C\gm \]
as claimed.

\end{proof}

\subsubsection*{Step 3: proving that $\Ass$ and $\Ex$ are mutually inverse}
Consider the map
\[ \preEx\preAss : I_A\ptp B \oplus A\ptp I_B \to I_A\ptp B \oplus A\ptp I_B \;.\]
Evaluating on elementary tensors, we find that
\[ \begin{aligned}
 \preEx\preAss \cvct{\sig_A(u\tp x)\tp b }{ a\tp\sig_B(v\tp y)}
 &  = \preEx\sig_C(u\tp b\tp x\tp\id[B]) + \preEx\sig_C(a\tp v\tp\id[A]\tp y) \\
 & = \cvct{\sig_A(u\tp x)\tp b}{0} + \cvct{0}{a\tp\sig_B(v\tp y)}  =  \cvct{\sig_A(u\tp x)\tp b }{ a\tp\sig_B(v\tp y)} \\
\end{aligned} \]
and so by continuity and linearity, $\preEx\preAss$ is the identity map on $I_A\ptp B \oplus A\ptp I_B$; in particular, $\Ex\Ass$ is the identity map on $\Coker(\tau_A\ptp\sid_B)\oplus\Coker(\sid_A\ptp\tau_B)$.

It remains only to show that the map $\preAss\preEx- \sid$ takes values in $\Image(\tau_C)$. Since $\sig_C$ surjects onto the domain of $\preEx$, it suffices to construct a bounded linear map $\rho: C\ptp C \to C\ptp C\ptp C$ such that
\begin{equation}\label{eq:Daffy}
\preAss\preEx\sig_C - \sig_C = \tau_C\rho
\end{equation}
which we do as follows. For any $a,x \in A$ and $b,y \in B$,
\[ \begin{aligned}
& (\preAss\preEx\sig_C-\sig_C) (a\tp b \tp x \tp y))  \\
 & = \preAss_A(\sig_A(a\tp x)\tp by) + \preAss_B(ax\tp\sig_B(b\tp y)) - \sig_C(a\tp b \tp x \tp y) \\
& = \sig_C (a\tp by\tp x\tp\id[B] + ax\tp b\tp\id[A]\tp y - a\tp b\tp x \tp y ) \\
& = \sig_C \bdy^C_1( (a\tp b) \tp (x\tp\id[B]) \tp (\id[A]\tp y)) \\
& = - \tau_C(a\tp b\tp x\tp\id[B]\tp \id[A]\tp y) \,.
\end{aligned} \]
We therefore define $\rho$ by the formula $\rho(a\tp b \tp x\tp y) = - a\tp b\tp x\tp\id[B]\tp\id[A]\tp y$\/. It is clear that $\rho$ is bounded linear, and by linearity and continuity we conclude that \eqref{eq:Daffy} holds. This completes Step~3.

\medskip
Theorem~\ref{t:qKahl_of_tp} now follows by combining Steps 1, 2 and 3.

\subsection*{Relation to the `Banach \Kahler\ module'}
This short section is not needed for the results to follow, but puts the seminormed module $\qKahl{A}$ into context.

Let $I_A^{[2]}$ denote the image of the product map $I_A\ptp I_A \to I_A$; note that this is {\it a priori}\/ strictly larger than $I_A^2=\lin\{ vw \st v,w \in I_A\}$, and is in general strictly smaller than~$\clos{I_A^2}$.

\begin{lemma}
$I_A^{[2]} = \Image(\tau_A)$.
\end{lemma}

\begin{proof}
We write $\pi_{I_A}$ for the product map $I_A\ptp I_A\to I_A$\/. Given $x_1,x_2, y_1, y_2 \in A$, we have
\[ \begin{aligned}
 \sig_A(x_1\tp x_2)\sig_A(y_1\tp y_2)
 & = (x_1\tp x_2 - x_1x_2\tp\id[A])(y_1\tp y_2 - y_1y_2\tp\id[A]) \\
 & = x_1y_1\tp x_2y_2 -x_1x_2y_1\tp y_2 - x_1y_1y_2\tp x_2 + x_1x_2y_1y_2\tp\id[A] \\
 & = \tau_A(x_1y_1\tp x_2 \tp y_2) \; .\\
\end{aligned}  \]
Let $\al : A^{\tp 4} \to A^{\ptp 3}$ be given by $\al(x_1\tp x_2\tp y_1\tp y_2)\defeq x_1y_1\tp x_2 \tp y_2$; then the preceding calculation shows that $\tau_A\al =\pi_{I_A}(\sig_A\ptp \sig_A)$. Since $A$ is unital $\al$ is surjective, and we conclude that
\[ I_A^{[2]} = \Image(\pi_{I_A}(\sig_A\ptp \sig_A)) = \Image(\tau_A\al) =\Image(\tau_A) \]
as required.
\end{proof}

From this the following corollary is immediate.
\begin{coroll}\label{c:Hff_of_qKahl}
The Hausdorff\-ifi\-cation of $\qKahl{A}$ is isomorphic, as a Banach $A$-module, to $I_A/\clos{I_A^2}$.
\end{coroll}

The point of this corollary is that the Banach $A$-module $I_A/\clos{I_A^2}$ has already been studied, in Runde's paper \cite{Run_Kahl}: it is the natural Banach analogue of the \dt{\Kahler\ module of differentials} for a commutative ring. 
Indeed, the statement and proof of our decomposition theorem for $\qKahl{A\ptp B}$\/ are modelled on the corresponding result and proof for the \Kahler\ module of a tensor product of rings. For example, the idea behind Theorem~\ref{t:qKahl_of_tp} is based on the `product rule' formula
\[ d_C(x\tp y) = (\id[A]\tp y)\cdot d_C(x\tp\id[B]) + (x\tp\id[B])\cdot d_C(\id[A]\tp y) \]
and the identification of $d_C(x\tp \id[B])$, $d_C(\id[A]\tp y)$ with $d_A(x)$ and $d_B(y)$ respectively.
\end{section}

\begin{section}{Hochschild homology via $\Tor^{A}$}
From here on, unless \emph{explicitly} stated otherwise, we let $A$ denote the Banach algebra $\lp{1}(\Z_+)$\/.
The following lemma is taken from the proof of \cite[Propn 7.3]{GouLyWh}.
\begin{lemma}\label{l:Fred's_lemma}
Let $q: \Ho{C}{1}(A,A) \to \lp{1}(\Nat)$ be the bounded linear map defined by $q(\id\tp\id)=0$ and
\[ q(z^k\tp z^l) =\frac{l}{k+l}z^{k+l} \quad(k,l\in\Z_+; k+l \geq 1).\]
Then $q$ is surjective and $\ker(q)=\Ho{B}{1}(A,A)$.
\end{lemma}
We note that the proof of this in \cite{GouLyWh} can be shortened slightly: see Appendix~\ref{app:HH1_YC} for the details.

\begin{coroll}\label{c:Ho1_l1Z+}
$\Ho{H}{1}(A,A)$ is a unit-linked, Banach $A$-module, whose underlying Banach space is isomorphic to $\lp{1}$.
\end{coroll}
\begin{proof}
First note that $\Ho{C}{1}(A,A)=\Ho{Z}{1}(A,A)$ (since $A$ is commutative). 

By Lemma~\ref{l:Fred's_lemma}, $\Ho{B}{1}(A,A)$ is a closed linear subspace of $\Ho{C}{1}(A,A)$ and the quotient space $\Ho{C}{1}(A,A)/\Ho{B}{1}(A,A)$ is isomorphic as a Banach space to $\Ho{C}{1}(A,A)/\ker(q) \iso \lp{1}(\Nat)$.

Moreover, $\Ho{B}{1}(A,A)$ is a submodule of the unit-linked $A$-module $\Ho{C}{1}(A,A)$: hence $\Ho{H}{1}(A,A)=\Ho{Z}{1}(A,A)/\Ho{B}{1}(A,A)=\Ho{C}{1}(A,A)/\Ho{B}{1}(A,A)$ is the quotient of a unit-linked Banach $A$-module by a closed submodule, and is thus itself a unit-linked Banach $A$-module as claimed.
\end{proof}

\begin{propn}\label{p:Ho1_l1Z+k}
Let $k\in\Nat$; let $A_1, \ldots, A_k$ denote copies of the Banach algebra $A=\lp{1}(\Z_+)$, and identify the convolution algebra $\fA_k=\lp{1}(\Z_+^k)$ with the tensor product $A_1\ptp\ldots\ptp A_k$.

Then $\Ho{H}{1}(\fA_k,\fA_k)$ is a symmetric, unit-linked, Banach $\fA_k$-bimodule, and we have an isomorphism of Banach $\fA_k$-modules
\[ \Ho{H}{1}(\fA_k,\fA_k) \iso \bigoplus_{i=1}^k A_1\ptp \ldots \ptp \Ho{H}{1}(A_i,A_i)\ptp\ldots \ptp A_k \]
In particular, the underlying Banach space of $\Ho{H}{1}(\fA_k,\fA_k)$ is isomorphic to $\lp{1}$.
\end{propn}
\begin{proof}
This is immediate from Corollaries~\ref{c:Ho1_l1Z+} and~\ref{c:baby_Kunn}.
\end{proof}

\begin{thm}\label{t:only_Harr}
Let $N$ be a unit-linked, symmetric $A$-bimodule and let $n\geq 1$. Then the canonical maps
\[ \begin{aligned} \HarH{H}{n}(A,N)  & \rTo \Ho{H}{n}(A,N) \\
 \HarC{H}{n}(A,N) & \rTo \Co{H}{n}(A,N) \end{aligned}\]
induce isomorphisms on homology and cohomology respectively. Moreover, there are isomorphisms of seminormed spaces
\[ \begin{aligned}
 \Ho{H}{n}(A,N) & \iso \Tor_{n-1}^A\left[N_R, \Ho{H}{1}(A,A)\right] & \iso \HarH{H}{n}(A,N) \\
 \Co{H}{n}(A,N) & \iso \Ext^{n-1}_A\left[\Ho{H}{1}(A,A),N_L\right] & \iso \HarC{H}{n}(A,N) \\
 \end{aligned} \]
\end{thm}

\begin{proof}
By \cite[Propn 7.3]{GouLyWh} the following facts hold:
\begin{itemize}
\item $\Ho{B}{1}(A,A)$ is a closed subspace of $\Ho{C}{1}(A,A)$;
\item the Banach space $\Ho{H}{1}(A,A)$ is isomorphic to $\lp{1}$;
\item the chain complex
\begin{equation}\label{diag:chaincomplex-is-exact}
 0 \lTo \Ho{H}{1}(A,A) \lTo^q \Ho{C}{1}(A,A) \lTo^{\bdy_1} \Ho{C}{2}(A,A)\lTo \ldots \end{equation}
is an exact sequence of Banach spaces.
\end{itemize}

We claim that
the complex \eqref{diag:chaincomplex-is-exact} is not merely exact, but is split exact in $\Ban$. This is proved inductively, as follows. Since $\Ho{H}{1}(A,A)$ is isomorphic as a Banach space to $\lp{1}$, the lifting property of $\lp{1}$-spaces with respect to open mappings allows us to find a bounded linear map $\rho_0: \Ho{H}{1}(A,A)\to \Ho{C}{1}(A,A)$ such that $q\rho_0=\sid$. Then since $\bdy_1$ surjects onto $\ker(q)$, and since $\Ho{C}{1}(A,A)$ is isomorphic as a Banach space to $\lp{1}$, the aforementioned lifting property of $\lp{1}$-spaces allows us to find a bounded linear map $\rho_1: \Ho{C}{1}(A,A) \to \Ho{C}{2}(A,A)$ such that $\bdy_1\rho_1=\sid - \rho_0 q$.
 Continuing in this way, at each stage using the fact that each $\Ho{C}{n}(A,A)$ is isomorphic to an $\lp{1}$-space, we may inductively construct bounded linear maps $\rho_n: \Ho{C}{n}(A,A) \to \Ho{C}{n+1}(A,A)$ such that $\bdy_n\rho_n+\rho_{n-1}\bdy_{n-1}=\sid$. 

Now let $\pi=\sid\ptp \BGS[\blob]{1}: \Ho{C}{\blob}(A,A)\to \HarH{C}{\blob}(A,A)$ be the BGS projection onto the Harrison summand. $\pi$ is a chain map, so we have a commuting diagram in $\LMod{A}$:
\begin{equation}
\begin{diagram}[tight,height=2em,width=4.5em]\label{diag:all-in-Harr}
0  \leftarrow & \Ho{H}{1}(A,A) & \lTo^q & \Ho{C}{1}(A,A) & \lTo^{\bdy_1} & \Ho{C}{2}(A,A) & \lTo^{\bdy_2} & \ldots \\
 & \dEq & & \dEq & & \dTo_{\pi_2} & &   \\
0 \leftarrow & \HarH{H}{1}(A,A) & \lTo_q & \HarH{C}{1}(A,A) & \lTo_{\bdy_1} & \HarH{C}{2}(A,A) & \lTo_{\bdy_2} &  \ldots 
\end{diagram}
\end{equation}
We have already observed that the top row of \eqref{diag:all-in-Harr} is split exact in $\Ban$. Since $\pi$ is a \emph{chain projection}, the bottom row is a direct summand of the top row and therefore (by a standard diagram-chase) must itself be split exact in $\Ban$.

\emph{Thus both rows are admissible resolutions of $\Ho{H}{1}(A,A)$ by $A$-projective Banach modules.} Since $\pi$ is left inverse to the inclusion chain map $\iota: \HarH{C}{*}(A,A)\to \Ho{C}{*}(A,A)$, the standard comparison theorem for projective resolutions tells us that $\iota\pi$ is chain homotopic to the identity. Therefore each of the induced chain maps
\[ \begin{diagram}
 N\ptpR{A}\Ho{C}{*+1}(A,A)  & \rTo^{\sid_N\ptpR{A}\pi} &  N\ptpR{A}\HarH{C}{*+1}(A,A)\\
 \lHom{A}\left( \Ho{C}{*+1}(A,A) , N \right)  & \lTo_{\lHom{A}(\pi, N)}  &  \lHom{A}\left( \HarH{C}{*+1}(A,A) , N \right)  \\
\end{diagram} \]
is chain homotopic to the identity, hence induces isomorphism on (co)homology.

Moreover, since $\Tor$ and $\Ext$ may be calculated using $A$-projective resolutions in the first variable,
\[ H_m\left[ N\ptpR{A}\Ho{C}{*+1}(A,A) \right] \iso \Tor_{m}^A\left[N_R, \Ho{H}{1}(A,A)\right] \iso H_m\left[ N\ptpR{A}\HarH{C}{*+1}(A,A) \right] \]
and
\[ H^m\left[ \lHom{A}\left( \Ho{C}{*+1}(A,A) , N \right) \right] \iso \Ext^m_A\left[\Ho{H}{1}(A,A),N_L\right] \iso H^m\left[ \lHom{A}\left( \HarH{C}{*+1}(A,A) , N \right) \right] \;.\]

By Proposition~\ref{p:CBA_Hoch} and Equations~\eqref{eq:UCT_BGSho}, \eqref{eq:UCT_BGSco}, there are chain isomorphisms
\[ \begin{aligned} \Ho{C}{*}(A,N) & \miso N\ptp_A \Ho{C}{*}(A,A) \\
\Co{C}{*}(A,N) & \miso \lHom{A} \left(\Ho{C}{*}(A,A), N\right)  \end{aligned} \]
and
\[ \begin{aligned} N\ptpR{A}\HarH{C}{*}(A,A) & \iso \HarH{C}{*}(A,N) \\
 \lHom{A}\left(\HarC{C}{*}(A,A), N\right)  & \iso \HarC{C}{*}(A,N)\;.
\end{aligned}\]
Under these chain isomorphisms we identify $\sid_N\ptpR{A}\pi$ with the BGS projection of $\Ho{C}{*}(A,N)$ onto $\HarH{C}{*}(A,N)$ and identify $\lHom{A}(\pi, N)$ with the inclusion of $\HarC{C}{*}(A,N)$ into $\Co{C}{*}(A,N)$. By the previous remarks both these maps induce isomorphism on (co)ho\-mology, and we are done.
\end{proof}

We shall build on this idea slightly to obtain partial results for cohomology of $\fA_k$. Our approach requires some results on the purely algebraic Hochschild homology groups $\Ho[\alg]{H}{*}(\sR_k, \sR_k)$, where $\sR_k$ denotes the polynomial algebra $\Cplx[z_1,\ldots,z_k]$.
\begin{thm}\label{t:polyring_HH}
Let $n \geq 2$. Then $\Ho[\alg]{H}{i,n-i}(\sR_k, \sR_k)=0$ for $1\leq i\leq n-1$.
\end{thm}
Informally, the theorem tells us that the simplicial homology of a polynomial algebra is confined to the Lie component.

\begin{remstar}
Theorem~\ref{t:polyring_HH} appears to be part of the folklore in commutative algebra and cohomology. The statement may be found in the remarks before \cite[Coroll~8.8.9]{Weibel} (though its proof is deferred to a later exercise). A proof of the special case $i=1$ (i.e.~for the Harrison summand, which is in fact all we will need) is given in \cite[Propn~3.1]{Barr_Harr}: first, one reduces the problem to one involving $\Ho[\alg]{H}{*}(\sR_k,\Cplx)$\/; then one applies a dimension-counting argument.
\end{remstar}

We shall use Theorem~\ref{t:polyring_HH}, combined with analytic results from \cite{GouJohWh} and \cite{GouLyWh}, to derive the analogous result for the simplicial homology of $\fA_k\iso\lp{1}(\Z_+^k)$. To pass between the algebraic and analytic settings we need a good way to approximate simplicial cycles on $\fA_k$ by simplicial cycles on $\sR_k$; this is done by establishing a suitable `density lemma' (Lemma~\ref{l:dense_alg-cyc} below).


Identify $\sR_k$ with the dense subalgebra of $\fA_k$ spanned by polynomials. The inclusion homomorphism $\sR_k \hookrightarrow \fA_k$ yields an inclusion of chain complexes
$\Ho[\alg]{C}{*}(\sR_k,\sR_k)\hookrightarrow \Ho{C}{*}(\fA_k,\fA_k)$\/. 
Identifying $\Ho{C}{n}(\fA_k, \fA_k)$ with $\lp{1}(\Z_+^k\times\ldots\times\Z_+^k)$, we see that $\Ho[\alg]{C}{n}(\sR_k,\sR_k)$ is dense in $\Ho{C}{n}(\fA_k,\fA_k)$ for each $n$.

We use multi-index notation, so that monomials in $\sR_k$ are written as $z^\al$ rather than $z_1^{\al_1}\dotsb z_n^{\al_n}$.
\begin{defn}
A \dt{monomial chain} in $\Ho[\alg]{C}{n}(\sR_k.\sR_k)$ is just a tensor of the form
\[ x= z^{\al(0)}\tp z^{\al(1)}\tp \ldots\tp z^{\al(n)} \]
where $\al(0),\al(1), \ldots, \al(n) \in\Z_+^k$. The \dt{total degree} of $x$ is the $k$-tuple {$\al(0)+\al(1)+\ldots+\al(n)$}, and is denoted by $\deg(x)$.

Given $N \in \Z_+^k$, we let $\pi^N_n:\Ho{C}{n}(\fA_k,\fA_k)\to\Ho{C}{n}(\fA_k,\fA_k)$ denote the norm-1 projection onto the closed linear span of the monomial chains with total degree $N$. More precisely, we define $\pi^N_n$ on monomial chains by
\[ \pi^N_n(x)\defeq  \left\{ \begin{aligned}
x & \quad\text{ if $\deg(x)=N$} \\
0 & \quad\text{ otherwise} \\
\end{aligned} \right\} \]
and extend by linearity and continuity.
\end{defn}
It is clear from this explicit definition that $\pi^N_n$ commutes with the action of $\sid\tp S_n$ on $\Ho{C}{n}(\fA_k,\fA_k)$, and hence commutes with each of the BGS idempotents $(\BGS[n]{i})_{i=1}^n$.

We \emph{claim} that $\pi^N_*$ is a chain map, i.e. that
 $\pi^N_n \bdy_n = \bdy_n\pi^N_{n+1}$ where $\bdy_n : \Ho{C}{n+1}(\fA_k,\fA_k) \to \Ho{C}{n}(\fA_k,\fA_k)$ is the Hochschild boundary map. Since $\bdy_n=\sum_{i\geq 0} (-1)^i \face[n]{i}$ is the alternating sum of face maps, it suffices to show that $\pi^N_n\face[n]{i}=\face[n]{i}\pi^N_{n+1}$ for each $i$. But this is immediate once we observe that each face map $\face[n]{i}:\Ho{C}{n+1}(\fA_k,\fA_k) \to \Ho{C}{n}(\fA_k,\fA_k)$ preserves the total degree of monomial chains, and so our claim is proved.

Given $N\in\Z_+^k$ and $n\in\Nat$, there are only finitely many monomial $n$-chains of degree~$N$; hence the range of $\pi^N_n$ is contained in $\Ho[\alg]{C}{n}(\sR_k,\sR_k)$. 
Therefore, for each $m \in \Nat$ we may define a chain projection $P^m_*: \Ho{C}{*}(\fA_k,\fA_k)\to \Ho{C}{*}(\fA_k,\fA_k)$ by
\[ P^m_n \defeq \sum_{N\in\Z_+^k\st \abs{N}\leq m} \pi^N_n  \]
By the remarks above, $P^m_n$ takes values in $\Ho[\alg]{C}{n}(\sR_k, \sR_k)$, and for every $n$-chain $x$ we have
\[ P^m_n x \to x \quad\text{ as $m \to \infty$\/.} \]
Moreover, $P^m$ commutes with the BGS projections.

\medskip
We now have everything in place for the following technical lemma.
\begin{lemma}[Density lemma]\label{l:dense_alg-cyc}
Let $1\leq i \leq n$ and let $x \in \Ho{Z}{i,n-i}(\fA_k,\fA_k)$. Then for every $\veps >0$ there exists $y \in \Ho[\alg]{Z}{i,n-i}(\sR_k,\sR_k)$ with $\norm{x-y}\leq \veps$.
\end{lemma}
\begin{proof}
We know that $P^m_n(x) \to x$ as $m\to\infty$. Choose $M$ such that $\norm{P^M_n(x)-x}\leq \veps$ and let $y\defeq P^M_n(x)\in\Ho[\alg]{C}{n}(\sR_k,\sR_k)$. Since $P^M$ is a chain map,
\[ \bdy y = \bdy P^M_n(x) = P^M_{n-1}\bdy(x) = 0 \]
and thus $y \in \Ho[\alg]{Z}{n}(\sR_k,\sR_k)$.  Finally,
\[ \BGS[n]{i}y =\BGS[n]{i} P^M_n(x) = P^M_n\BGS[n]{i}(x) = P^M_n(x) = y \]
and thus $y$ has BGS type $(i,n-i)$ as required.
\end{proof}

\begin{propn}[Simplicial homology confined to Lie component]
\label{p:no_higher_non-Lie}
Let $n \geq 2$ and let $1\leq i\leq n-1$. Then $\Ho{H}{i,n-i}(\fA_k,\fA_k)=0$.
\end{propn}
\begin{proof}
By \cite[Thm 7.5]{GouLyWh} we know that the boundary maps on the Hochschild chain complex $\Ho{C}{*}(\fA_k,\fA_k)$ are open mappings. Let $C$ be the constant of openness of the boundary map $\bdy_n:\Ho{C}{n+1}(\fA_k,\fA_k)\to\Ho{C}{n}(\fA_k,\fA_k)$.

Fix $\veps \in (0,1)$ and let $x\in\Ho{Z}{i,n-i}(\fA_k,\fA_k)$.

\medskip\noindent{\bf Claim:} there exists $\gm \in \Ho[\alg]{C}{i,n+1-i}(\sR_k, \sR_k)$ with $\norm{\gm}\leq C(1+\veps)^2\norm{x}$ and $\norm{x-\bdy\gm}\leq \veps\norm{x}$.
\medskip

Assuming that the claim holds, a standard inductive approximation argument may be used to produce $u \in \Ho[\alg]{C}{i,n+1-i}(\sR_k,\sR_k)$ with $\norm{u}\leq (1-\veps)^{-1}(1+\veps)^2 C\norm{x}$ and $\bdy u=x$; in particular $x\in\Ho{B}{i,n-i}(\fA_k,\fA_k)$. Since $x$ was an arbitrary cycle of type $(i,n-i)$, this shows that $\Ho{Z}{i,n-i}(\fA_k,\fA_k)=\Ho{B}{i,n-i}(\fA_k,\fA_k)$ as required.

It therefore suffices to prove that we can find such a $\gm$, which we do as follows. By our density lemma~\ref{l:dense_alg-cyc} we know there exists $y \in \Ho[\alg]{Z}{i,n-i}(\sR_k,\sR_k)$ with $\norm{x-y}\leq\veps\norm{x}$. By Theorem~\ref{t:polyring_HH}, $y=\bdy w$ for some $(n+1)$-chain $w$ on $\sR_k$. Regard $w$ as an element of $\Ho{C}{n+1}(\fA_k,\fA_k)$: since $\bdy_n$ is open with constant $C$ there exists an $(n+1)$-chain $\gm$ on $\fA_k$ such that $\bdy\gm=\bdy w= y$ and $\norm{\gm}\leq C(1+\veps)\norm{y}\leq C(1+\veps)^2\norm{x}$. This proves our claim and hence concludes the proof of the theorem.
\end{proof}

\begin{lemma}\label{l:induced_module}
Let $B$, $C$ be unital Banach algebras, let $M$ be a left Banach $B\ptp C$-module, let $X$ be a left Banach $C$-module and let $M_C$ be the left Banach $C$-module obtained by letting $C$ act via the homomorphism $C \to B\ptp C, c \mapsto \id[B]\tp c$.

Then for each $n$,
\[ \Ext_{B\ptp C}^n (B\ptp X, M) \iso \Ext_C^n (X, M_C) \]
\end{lemma}
\begin{proof}
Let $0\leftarrow X \leftarrow P_\blob$ be the standard \dt{bar resolution} of $X$ by left $C$-projective modules (see \cite[Propn 2.9]{Hel_HBTA}). This complex is split exact in $\Ban$\/: hence, by functoriality of $B\ptp\blank: \Ban \to \LunMod{B}$, the complex $0\leftarrow B\ptp X  \leftarrow B\ptp P_\blob$ is an admissible complex of Banach $B$-modules and module maps. Moreover, since $B$ is unital, it is easily checked that $B\ptp P_n$ is $B\ptp C$-projective for every $n$. Thus $B\ptp P_\blob$ is an admissible $B\ptp C$-projective resolution of $B\ptp X$\/, and so
\[ \begin{aligned}
\Ext_{B\ptp C}^n (B\ptp X, M) & \iso H^n\left[ \lHom{B\ptp C} (B\ptp P_\blob, M) \right] \\
& = H^n\left[ \lHom{C}(P_\blob, M) \right] 
& \iso \Ext_C^n (X, M_C) \\
\end{aligned}  \]
as claimed.
\end{proof}

We can now prove the main result of this paper. As in the statement of Proposition~\ref{p:Ho1_l1Z+k}, let us identify $\fA_k$ with the $k$-fold tensor product $A_1\ptp \ldots \ptp A_k$, where each $A_i$ denotes a copy of the Banach algebra $A$.

\begin{thm}\label{t:Harr_of_Ak}
Let $M$ be a unit-linked symmetric $\fA_k$-bimodule. For each $i=1,\ldots, k$ the inclusion of $A_i$ into $\fA_k$ induces an $A_i$-bimodule structure on $M$; denote the resulting symmetric $A_i$-bimodule by $M_i$. Then for $n \geq 1$,
\[ \HarC{H}{n}(\fA_k,M)\iso \bigoplus_{i=1}^k \Ext_{A_i}^{n-1}(\Ho{H}{1}(A_i,A_i),M_i) \iso \bigoplus_{i=1}^k \Co{H}{n}(A_i,M_i) \]
\end{thm}

\begin{proof}
The second isomorphism follows from Theorem~\ref{t:only_Harr}, so we need only verify the first one. This is done using Proposition~\ref{p:Harr-as-Ext}, following a procedure very similar to that in the proof of Theorem~\ref{t:only_Harr}.

Consider the Hochschild chain complex $\Ho{C}{\blob}(\fA_k,\fA_k)$. By Proposition~\ref{p:no_higher_non-Lie} all the homology has to live in the Lie component of the Hodge decomposition: in particular, the Harrison summand
\[ \HarH{C}{1}(\fA_k,\fA_k) \lTo^{\bdy_1} \HarH{C}{2}(\fA_k,\fA_k) \lTo^{\bdy_2}\ldots\]
is an exact sequence in $\Ban$. The cokernel of $\bdy_1$ is $\HarH{H}{1}(\fA_k,\fA_k)$ and by Proposition~\ref{p:Ho1_l1Z+k} this is a Banach space isomorphic to~$\lp{1}$. Hence
\[ 0\lTo \HarH{H}{1}(\fA_k,\fA_k)\lTo^q \HarH{C}{1}(\fA_k,\fA_k) \lTo^{\bdy_1} \ldots\]
is an exact sequence in $\Ban$ with every term isomorphic to a complemented subspace of $\lp{1}$: the lifting property of such spaces with respect to surjective linear maps now allows us to inductively construct a \emph{splitting in $\Ban$} for this exact sequence.

Thus the conditions of Proposition~\ref{p:Harr-as-Ext} are satisfied, and using that proposition we obtain an isomorphism of seminormed spaces
\[ \HarC{H}{n}(\fA_k,M) \iso \Ext^{n-1}_{\fA_k}\left( \Ho{H}{1}(\fA_k,\fA_k), M\right) \;.\]
By Proposition~\ref{p:Ho1_l1Z+k}
\[ \Ho{H}{1}(\fA_k,\fA_k) \iso \bigoplus_{i=1}^k A_1\ptp \ldots \ptp \Ho{H}{1}(A_i,A_i)\ptp\ldots \ptp A_k \quad;\]
so by Lemma~\ref{l:induced_module} we have, for each $i$\/,
\[ \Ext^{n-1}_{\fA_k}\left[ \left(\bigptp_{j\neq i}A_j \right)\ptp \Ho{H}{1}(A_i,A_i), M\right] \iso \Ext^{n-1}_{A_i}(\Ho{H}{1}(A_i,A_i), M_i) \;.\]
This implies that
\[  \HarC{H}{n}(\fA_k,M) \iso \bigoplus_{i=1}^k \Ext^{n-1}_{A_i}(\Ho{H}{1}(A_i,A_i), M_i) \]
and our proof is complete.
\end{proof}
\begin{remstar}
The proof of Theorem~\ref{t:Harr_of_Ak} can be easily modified to yield a parallel result for Harrison \emph{homology} of $\fA_k$, as follows: using the same notation as above, we have
\[ \HarH{H}{n}(\fA_k,M)\iso \bigoplus_{i=1}^k \Tor^{A_i}_{n-1}(\Ho{H}{1}(A_i,A_i),M_i) \iso \bigoplus_{i=1}^k \Ho{H}{n}(A_i,M_i) \]
for all $n \geq 1$. We omit the details.
\end{remstar}

\end{section}

\begin{section}{Calculation of some second cohomology groups}\label{s:someH2}
Our hope is that Theorem~\ref{t:Harr_of_Ak} can be used as a unifying tool in the calculation of various cohomology groups of $\fA_k$\/. As an illustration, we shall in this section use it to identify $\Co{H}{2}(\fA_k,\fA_k)$ with a certain infinite-dimensional Banach space of derivations.

\begin{rem}
As already mentioned, in the case $k=1$ it has long been known that this cohomology group is non\-zero, and a direct argument to show it is Hausdorff can be found in \cite{DaDunc}: see the remarks there after Equation (1.14). It is also mentioned in \cite{DaDunc} that similar results should hold for $k\geq 2$\/. Thus the novelty of this section is not so much the result itself (although our version appears to be the first explicit statement and proof in the literature). Rather, it lies in our attempt to attack these problems in a systematic way that might generalise to higher-degree cohomology.
\end{rem}

We first sketch how our proof goes in the case $k=1$. As in the previous section, $A$ will denote $\lp{1}(\Z_+)$\/; it is also convenient to denote the Banach algebra $\lp{1}(\Z)$ by $C$\/. The key idea is that the short exact sequence of (symmetric) Banach $A$-modules
\[ 0 \to A\to C\to C/A \to 0 \]
gives rise to a long exact sequence of cohomology
\[ \Co{H}{1}(A,C)
 \rTo \Co{H}{1}(A,C/A)
  \rTo \Co{H}{2}(A,A)
\rTo \Co{H}{2}(A,C) \]
and the two end terms in this sequence turn out to be zero.

For general $k$ one uses Theorem~\ref{t:Harr_of_Ak}, loosely speaking, to turn a $k$-variable problem into a direct sum of $k$-copies of the one-variable problem, to which the argument just sketched applies. The precise statement requires some notation: regarding $C$ as a sym\-metric $A$-bimodule in the obvious way, and regarding $A$ as a closed submodule of $C$, we may form the quotient $A$-bimodule $Q=C/A$; then for any Banach space $E$\/, we regard $Q\ptp E$ as a Banach $A$-bimodule by letting $A$ act on the first factor.

\begin{thm}\label{t:H2AkAk}
There are isomorphisms of seminormed spaces
\begin{equation}
\Co{H}{2}(\fA_k,\fA_k)\iso\HarC{H}{2}(\fA_k,\fA_k) \iso \left( \Co{Z}{1}(A, Q\ptp\lp{1}(\Z_+^{k-1})) \right)^{\oplus k}
\end{equation}
In particular, $\Co{H}{2}(\fA_k,\fA_k)$ is an infinite-dimensional Banach space.
\end{thm}

The proof will, in addition to using Theorem~\ref{t:Harr_of_Ak}, require some preliminary results which may be known to specialists but which we give for sake of completeness.

\begin{propn}\label{l:inherit-biinj}
Let $N$ be a Banach $C$-bimodule, regarded as a Banach $A$-bimodule via the inclusion homomorphism $A\hookrightarrow C$\/. Then
\begin{itemize}
\item[$(i)$] $N$ is $A$-biflat;
\item[$(ii)$] $\Co{H}{n}(A,N')=0$ for all $n\geq 1$\/.
\end{itemize}
\end{propn}

\begin{proof}
First note that assertion $(ii)$ follows from assertion $(i)$, since the dual of a biflat module is bi-injective and so by \cite[Thm~4.7]{Hel_HBTA} 
\[ \Co{H}{n}(A,N') \iso \Ext^n_{A^e}(A, N') =0 \quad\text{ for $n\geq 1$\/.} \]
Hence it remains only to prove $(i)$, or equivalently, to prove that $N'$ is $A$-bi-injective.

This will follow once we construct a bounded linear map $\rho: \Lin{A}{N'} \to N'$ such that
\begin{itemize}
\item for every $p\in\Z_+$ and $T\in\Lin{A}{N'}$,
\begin{equation}\label{eq:jackanory}
  \rho(z^p\cdot T)= z^p\cdot\rho(T) \;;
\end{equation}
\item for every $\psi\in N'$, $\rho(J\psi)=\psi$, where $J\psi:A\to N'$ is defined by
\[ \pair{J\psi(a)}{y} \defeq \pair{\psi}{y\cdot a} \qquad(a\in A, y\in N)\/.\]
\end{itemize}

Fix a Banach limit $\mathop{\rm LIM}$ on $\lp{\infty}(\Nat)$\/: then, for 
each $T\in \Lin{A}{N'}$ and $y\in N'$, let
\begin{equation}\label{eq:bagpuss}
 \pair{\rho(T)}{y}\defeq \mathop{\rm LIM}\nolimits_n \pair{T(z^n)}{y\cdot z^{-n}} \end{equation}
where the right-hand side is well-defined since $N$ is a $C$-module.
Linearity and continuity of $\mathop{\rm LIM}$ imply that the formula \eqref{eq:bagpuss} defines a bounded linear map $\rho:\Lin{A}{N'}\to N'$\/: and translation-invariance of $\mathop{\rm LIM}$ implies that Equation~\eqref{eq:jackanory} holds. Finally, since $\mathop{\rm LIM}$ sends the constant sequence $(1,1,\ldots)$ to $1$\/, it is easily checked that $\rho(J\psi)=\psi$ for every $\psi\in N'$\/.
\end{proof}

\begin{rem}
At a more abstract level, this proof works because $C$ is \emph{amenable} (so that every $C$-bimodule is $C$-biflat) and because $C$ is itself flat as an $A$-module. A more systematic approach to this phenomenon is given in \cite{MCW_UA}: see \S4 in particular.
\end{rem}

Note that $C\ptp\lp{1}(\Z_+^{k-1})$ is itself a dual $C$-bimodule, with predual $c_0(\Z\times\Z^{k-1})$ where $C$ acts by translation `in the first variable'. Hence
\begin{equation}\label{eq:YANA}
\Co{H}{1}(A, C\ptp\lp{1}(\Z_+^{k-1})) = \Co{H}{2}(A, C\ptp\lp{1}(\Z_+^{k-1})) = 0
\end{equation}

\begin{proof}[Proof of Theorem \ref{t:H2AkAk}]
For \emph{this proof} let us temporarily write $E$ for the Banach space $\lp{1}(\Z_+^{k-1})$\/.

The first isomorphism follows from the following observations: $\Co{Z}{1,1}(\fA_k,\fA_k)$ is the space of all bounded, anti\-symmetric $2$-cocycles $\fA_k\times\fA_k\to\fA_k$ (see Theorem~\ref{t:BGSalt} above). By \cite[Thm~2.3]{BEJ_high}, any such $2$-cocycle must be a derivation in each variable; but by the Singer-Wermer theorem (or a direct argument) the only bounded derivation from $\fA_k$ to itself is the zero map. Hence $\Co{H}{1,1}(\fA_k,\fA_k)=\Co{Z}{1,1}(\fA_k,\fA_k)=0$\/.

To prove the second isomorphism, we invoke Theorem~\ref{t:Harr_of_Ak} to obtain an isomorphism of seminormed spaces
\[ \HarC{H}{2}(\fA_k,\fA_k) \iso \bigoplus_{i=1}^k \Ext^2_A( \Ho{H}{1}(A,A), M_i) \]
where for each $i$\/, $M_i$ denotes the $A$-bimodule obtained by letting $A$ act on $\lp{1}(\Z_+^k)$ by `multiplication in the $i$th variable'\/.
By symmetry it is clear that $M_1,\ldots, M_k$ are all isomorphic as Banach $A$-bimodules to $A\ptp E$, and so to complete the proof it suffices to show that
\begin{equation}
\Co{H}{2}(A,A\ptp E) \iso \Co{Z}{1}(A,Q\ptp E)
\end{equation}

The short exact sequence of $A$-bimodules
\[ 0\to A \to C \to Q \to 0 \]
is admissible (splits in $\Ban$), and so remains an admissible short exact sequence of $A$-bimodules when we tensor with the Banach space $ E$\/. Hence we have a long exact sequence of cohomology
\[ \Co{H}{1}(A,C\ptp E)
 \rTo \Co{H}{1}(A, Q\ptp E)
  \rTo^{\mathop{\rm conn}} \Co{H}{2}(A,A\ptp E)
\rTo \Co{H}{2}(A,C\ptp E) \]
By Equation \eqref{eq:YANA} the two end terms are zero and hence the `connecting homomorphism' $\mathop{\rm conn} : \Co{H}{1}(A, Q\ptp E) \to \Co{H}{2}(A, A\ptp E)$ is bijective; by \cite[Lemma~0.5.9]{Hel_HBTA}, $\mathop{\rm conn}$ is therefore an \emph{isomorphism of seminormed spaces}. Finally, since $Q\ptp E$ is a symmetric bimodule, $\Co{H}{1}(A,Q\ptp E)=\Co{Z}{1}(A, Q\ptp E)$ and the proof is complete. 
\end{proof}
\end{section}

\appendix
\begin{section}{Another proof that $\Ho{H}{1}(\lp{1}(\Z_+),\lp{1}(\Z_+))$ is an $\lp{1}$-space}\label{app:HH1_YC}

See Lemma \ref{l:Fred's_lemma} above for the precise statement. The
proof of this result in \cite[Propn 7.3]{GouLyWh} is somewhat fiddly. We present a slightly more streamlined approach which appears to be new.

\begin{proof}
Let $q:\Ho{C}{1}(A,A)\to \lp{1}(\Nat)$ be defined as above. We define bounded linear maps $\sB$, $\sf S$ and $\sf H$ as follows.
\[ \sB(z^N) \defeq 1\tp z^N \]
\begin{flalign*}
{\sf S}(z^{N-j}\tp z^j) & \defeq \left\{
\begin{array}{lr}
\id\tp z^j\tp z^{N-j}  + z^j\tp z^j\tp z^{N-2j} & \quad\text{ if $0\leq j\leq N/2$} \\
\id\tp z^j\tp z^{N-j}  - z^{N-j}\tp z^{2j-N}\tp z^{N-j} &  \quad\text{ if $N/2 \leq j\leq N$}
\end{array}\right. \\
\end{flalign*} 
\begin{flalign*}
{\sf H}(z^{N-j}\tp z^j) & \defeq \left\{
	\begin{aligned}
	 2z^{N-j}\tp z^j + z^{2j}\tp z^{N-2j} & &\quad\text{ if $0\leq j\leq N/2$} \\
	 2z^{N-j}\tp z^j - z^{2N-2j}\tp z^{2j-N} & &\quad\text{ if $N/2 \leq j\leq N$}
	\end{aligned} \right.
\end{flalign*} 

\medskip
\noindent{\bf Claim.}
The maps $q$, $\sB, {\sf H}$ fit into a diagram
\[ \begin{diagram}[tight,height=2em]\label{diag:predual}
\lp{1}(\Nat) & \LR{q}{\sB} & \Ho{C}{1}(A,A) & \lTo^{\bdy} & \Ho{C}{2}(A,A) \\
  & & \uTo^{{\sf H}} & \ruTo^{{\sf S}} & \\
  & & \Ho{C}{1}(A,A) & &
\end{diagram} \]
where $q\sB=\sid$ and
\[ (\sid-\sB q){\sf H} = \bdy{\sf S} \]
(Here $\bdy$ denotes the Hochschild boundary operator.)

\medskip
The claim can be proved by direct checking on elementary tensors. Now observe that since $\norm{2\sid-{\sf H}}\leq 1$, ${\sf H}$ is \emph{invertible} as a bounded linear operator on the Banach space $\Ho{C}{1}(A,A)$. Hence $\sid-\sB q = \bdy{\sf S}{\sf H}^{-1}$ and the complex
\[ \lp{1}(\Nat) \lTo^q \Ho{C}{1}(A,A) \lTo^{\bdy} \Ho{C}{2}(A,A) \]
is thus split exact in $\Ban$.
\end{proof}

\vspace{0.5em}
\begin{remstar}
The argument just given may seem slightly mysterious, as we have provided no explanation of how one might come up with the maps $\sB$ and ${\sf S}$. In fact the construction above was discovered while considering the dual problem of proving that $\Co{H}{2}(A,A')$ is a Banach space. Further details can be found in Appendix C of the author's thesis~\cite{YC_PhD}.
\end{remstar}

\end{section}


\begin{thebibliography}{10}

\bibitem{Barr_Harr}
M.~Barr.
\newblock Harrison homology, {H}ochschild homology and triples.
\newblock {\em J. Algebra}, 8:314--323, 1968.

\bibitem{YC_PhD}
Y.~Choi.
\newblock {\em Cohomology of commutative {B}anach algebras and
  {$\ell^1$}-semigroup algebras}.
\newblock PhD thesis, University of Newcastle upon Tyne, 2006.

\bibitem{DaDunc}
H.~G. Dales and J.~Duncan.
\newblock Second-order cohomology groups of some semigroup algebras.
\newblock In {\em Banach algebras '97 (Blaubeuren)}, pages 101--117. de~Gruyter, Berlin, 1998.

\bibitem{DefFlor}
A.~Defant and K.~Floret.
\newblock {\em Tensor norms and operator ideals}, volume 176 of {\em
  North-Holland Mathematics Studies}.
\newblock North-Holland Publishing Co., Amsterdam, 1993.

\bibitem{G_Barr}
M.~Gerstenhaber.
\newblock Developments from {B}arr's thesis.
\newblock {\em J. Pure Appl. Algebra}, 143(1-3):205--220, 1999.

\bibitem{GS_Hodge}
M.~Gerstenhaber and S.~D. Schack.
\newblock A {H}odge-type decomposition for commutative algebra cohomology.
\newblock {\em J. Pure Appl. Algebra}, 48(3):229--247, 1987.

\bibitem{GouJohWh}
F.~Gourdeau, B.~E. Johnson, and M.~C. White.
\newblock The cyclic and simplicial cohomology of {$\ell^1({\mathbb N})$}.
\newblock {\em Trans. Amer. Math. Soc.}, 357(12):5097--5113 (electronic), 2005.

\bibitem{GouLyWh}
F.~Gourdeau, Z.~A. Lykova, and M.~C. White.
\newblock A {K}\"unneth formula in topological homology and its applications to
  the simplicial cohomology of {$\ell^1({\mathbb Z}^k_+)$}.
\newblock {\em Studia Math.}, 166(1):29--54, 2005.

\bibitem{HarrCo_62}
D.~K. Harrison.
\newblock Commutative algebras and cohomology.
\newblock {\em Trans. Amer. Math. Soc.}, 104:191--204, 1962.

\bibitem{Hel_HBTA}
A.~{\relax Ya}. Helemskii.
\newblock {\em The homology of {B}anach and topological algebras}, volume~41 of
  {\em Mathematics and its Applications (Soviet Series)}.
\newblock Kluwer Academic Publishers Group, Dordrecht, 1989.

\bibitem{BEJ_high}
B.~E. Johnson.
\newblock Higher-dimensional weak amenability.
\newblock {\em Studia Math.}, 123(2):117--134, 1997.

\bibitem{Lod_CH}
J.-L. Loday.
\newblock {\em Cyclic homology}, volume 301 of {\em Grundlehren der
  Mathematischen Wissenschaften [Fundamental Principles of Mathematical
  Sciences]}.
\newblock Springer-Verlag, Berlin, 1992.
\newblock Appendix E by Mar\'\i a O. Ronco.

\bibitem{Run_Kahl}
V.~Runde.
\newblock A functorial approach to weak amenability for commutative {B}anach algebras.
\newblock {\em Glasgow Math. J.}, 34(2):241--251, 1992.

\bibitem{Ryan_TP}
R.~A. Ryan.
\newblock {\em Introduction to tensor products of {B}anach spaces}.
\newblock Springer Monographs in Mathematics. Springer-Verlag London Ltd.,
  London, 2002.

\bibitem{Weibel}
C.~A. Weibel.
\newblock {\em An introduction to homological algebra}, volume~38 of {\em
  Cambridge Studies in Advanced Mathematics}.
\newblock Cambridge University Press, Cambridge, 1994.

\bibitem{MCW_UA}
M.~C. White.
\newblock Injective modules for uniform algebras.
\newblock {\em Proc. London Math. Soc. (3)}, 73(1):155--184, 1996.

\end{thebibliography}

\vfill

\noindent%
\begin{tabular}{l@{\hspace{32mm}}l}
{\bf Address:}  &  {\bf Current address:}\\
Department of Mathematics, & 
	D\'epartement de math\'ematiques\\
	& \text{\hspace{1.0em}} et de statistique, \\
Machray Hall & Pavillon Alexandre-Vachon \\
University of Manitoba &
	Universit\'e Laval \\
Winnipeg, MB & Qu\'ebec, QB \\
Canada, R3T 2N2 & Canada, G1V 0A6 \\
	& \\
	& {\bf Email: \tt y.choi.97@cantab.net}
\end{tabular}

\end{document}